\documentclass[10pt]{amsart}     
\usepackage{amssymb}      
\usepackage{graphicx}
\usepackage{a4wide}
\usepackage{ytableau}
\usepackage{comment}
\usepackage{xcolor}
\usepackage{pgf,tikz}
\usetikzlibrary{arrows.meta}
\usetikzlibrary{calc}
\usetikzlibrary{arrows, automata}
\usepackage{subcaption} 
\usepackage{mathtools} 
\theoremstyle{definition} 
\newtheorem{lemma}{Lemma}[section] 
\newtheorem{corollary}[lemma]{Corollary} 
\newtheorem{definition}[lemma]{Definition}
\newtheorem{problem}[lemma]{Problem}
\newtheorem{theorem}[lemma]{Theorem} 
 
\newtheorem{proposition}[lemma]{Proposition}
\newtheorem{remark}[lemma]{Remark} 
\newtheorem{example}[lemma]{Example} 
\usepackage[backend=biber,maxbibnames=5]{biblatex}
\DeclareFieldFormat{doi}{\href{https://dx.doi.org/#1}{\textsc{doi}}}
\DeclareFieldFormat{url}{\href{#1}{\textsc{url}}}
\renewbibmacro*{doi+eprint+url}{
  \printfield{doi}
  \newunit\newblock
  \iftoggle{bbx:eprint}{
    \usebibmacro{eprint}
  }{}
  \newunit\newblock
  \iffieldundef{doi}{
    \usebibmacro{url+urldate}}
  {}
}
\renewbibmacro{in:}{}
\AtEveryBibitem{
  \ifentrytype{book}
  {\clearfield{pages}}
}

\usepackage[colorlinks=true,linkcolor=blue,citecolor=magenta]{hyperref} 

\makeatletter
\@namedef{subjclassname@2020}{\textup{2020} Mathematics Subject Classification}     
\makeatother

\def\QQ{\mathbb{Q}}
\def\CC{\mathbb{C}}
\def\oc{\mathrm{OC}}
\def\xx{{\bf x}}

\title[Straight polyomino tilings and special rim-hook tableaux]{Straight polyomino tilings of rectangles\\ and special rim-hook tableaux}  
\author{Mudit Aggarwal}
\address{University of British Columbia, Vancouver, British Columbia, Canada.} 
\email{muditagg@math.ubc.ca}

\author{Hrishik Koley}
\address{Indian Statistical Institute, Bengaluru, India.}
\email{hrishik.math@outlook.com}

\author{Samrith Ram}
\address{Indraprastha Institute of Information Technology Delhi (IIIT-Delhi), New Delhi 110020, India.}
\email{samrithram@gmail.com}

\begin{document}
\maketitle
\begin{abstract}
We derive explicit rational generating functions for weighted tilings of $2k\times n$ rectangles by straight $k\times 1$ tiles. Our approach combines a decomposition by fault lines with a Hadamard-product framework. Tools from algebraic combinatorics are used together with an identity of Littlewood on Schur expansions of plethystic compositions of elementary symmetric functions. This translates the tiling problem into a combinatorial framework via special rim-hook tableaux. On the tiling side, Graham's theorem on fault-free tilings provides the key input needed to complete the analysis.
\end{abstract}
\tableofcontents 
\section{Introduction}\label{sec:introduction}

The exact enumeration of tilings of planar regions is classical, but completely explicit formulas are known only in special families. For domino tilings of $m\times n$ rectangles, the reduction to perfect matchings yields the exact product formulas of Temperley and Fisher \cite{TemperleyFisher1961} and Kasteleyn~ \cite{Kasteleyn63}. Another benchmark family is provided by Aztec diamonds: Elkies, Kuperberg, Larsen, and Propp~\cite{ElkiesKuperbergLarsenPropp1992} proved that the Aztec diamond of order $n$ has exactly $2^{n(n+1)/2}$ domino tilings. These examples illustrate how exceptional closed product formulas are, even for highly structured regions.

When one dimension is fixed, exact information is often encoded by rational generating functions or linear recurrences rather than product formulas. Read \cite{Read1980} developed transfer-matrix techniques for domino strip tilings. Klarner and Pollack \cite{KlarnerPollack1980} computed fixed-width domino generating functions and gave an algorithm for producing rational forms, while Hock and McQuistan \cite{HockMcQuistan1984} derived explicit recurrences in the strip length for additional small-width cases. Stanley \cite{Stanley1985} later showed that if $A_{m,n}$ denotes the number of domino tilings of a $m\times n$ rectangle, then $\sum_{n\geq 0}A_{m,n}x^n$ is rational and analyzed its numerator and denominator in detail. Propp \cite{Propp2001} subsequently proved a reciprocity theorem for these fixed-width numbers, showing that the same linear recurrence extends to negative indices in a controlled way.

Other exact rectangle-enumeration results arise from additional local restrictions. Ruskey and Woodcock \cite{RuskeyWoodcock2009} obtained ordinary generating functions for fixed-height tatami tilings of rectangles and explicit formulas as sums of binomial coefficients. Erickson, Ruskey, Woodcock, and Schurch \cite{EricksonRuskeyWoodcockSchurch2011} extended this circle of ideas to monomer-dimer tatami tilings, proving that for fixed height, the generating function is again rational. Exact enumeration becomes accessible using the transfer-matrix approach when one dimension is fixed and the admissible local structure is sufficiently rigid.

Complexity-theoretic results help explain why such exact formulas are uncommon in full generality. In Valiant's \cite{val79complexity} counting-complexity framework, for any fixed finite tile set, the decision problem of tileability of a finite input region is in NP (a tiling is a polynomial-size certificate), while the corresponding counting problem is a function in \#P. Moore and Robson \cite{MooreRobson2001} proved NP-completeness for tileability of finite square-lattice regions by right trominoes. Pak and Yang~\cite{PakYang2013} later showed that there exists a finite set of rectangles for which tileability of simply connected regions is NP-complete, and counting such tilings is \#P-complete. In dimension three, Chan and Pak \cite{ChanPak24} proved that deciding whether two regions have the same number of domino tilings is not in the polynomial hierarchy unless that hierarchy collapses to a finite level. These hardness results concern variable input regions; by contrast, we study the explicit counting problem on the structured rectangle family $2k\times n$.

Among straight polyomino tilings of rectangles, Klarner \cite[Thm.~5]{MR248643} proved a basic divisibility criterion: a rectangle can be tiled by $k\times 1$ bars if and only if one of its side lengths is divisible by $k$. Kreweras \cite{Kreweras1995} derived a sixth-order recurrence for tilings of a $3\times n$ rectangle by trominoes when both tromino shapes are allowed. Graham \cite{Graham1981} studied the existence of fault-free tilings of rectangles by coprime rectangles. More recently, Aggarwal and Ram \cite{MR4537778} obtained explicit generating functions for tilings of $m\times n$ rectangles by $k\times 1$ tiles in the narrow regime $m<2k$. Using Aggarwal and Ram \cite[Theorem~8]{MR4537778} it is straightforward to show that for $k\leq m<2k$, the number of tilings of an $m\times k\ell$ rectangle using $k\times 1$ tiles equals
\begin{equation}
\sum_{j=0}^{\ell}(m-k+1)^j\binom{kj+\ell-j}{\ell-j}.
\end{equation}
To the best of our knowledge, no other general exact enumeration formulas are known for rectangle tilings by straight bars of length greater than two. This article treats the boundary case $m=2k$, which lies just beyond that range, and gives an explicit weighted generating function in that setting.

Let $\mathcal{T}(m,n;k)$ denote the collection of all tilings of an $m\times n$ rectangle with $k\times 1$ tiles (called straight $k$-polyominoes or $k$-bars). By Klarner's divisibility criterion, the set $\mathcal{T}(m,n;k)$ is nonempty precisely when $k$ divides either $m$ or $n$. To each tiling $T\in \mathcal{T}(m,n;k)$, we associate a weight $w(T):=a^vb^h$, where $v$ and $h$ are the numbers of vertical and horizontal tiles in $T$. Define
\begin{align*}
  t(m,n;k):=\sum_{T\in \mathcal{T}(m,n;k)}w(T),
\end{align*}
and note that $t(m,n;k)$ is a homogeneous polynomial in $a$ and $b$ of degree $mn/k$.

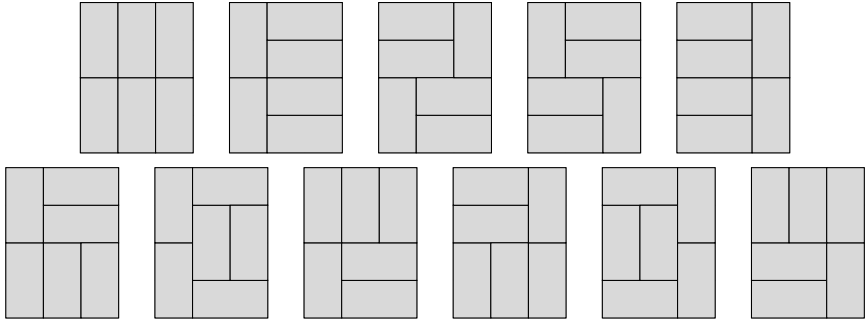
\begin{figure}
  \centering
  \tikzset{
    dominoH/.style={fill=gray!30,draw=black},
    dominoV/.style={fill=gray!30,draw=black}
  }

  \begin{tikzpicture}[scale=0.5]
    \draw[step=1cm,gray] (0,0) grid (3,4);
    \fill[dominoV] (0,0) rectangle (1,2);
    \fill[dominoV] (0,2) rectangle (1,4);
    \fill[dominoV] (1,0) rectangle (2,2);
    \fill[dominoV] (1,2) rectangle (2,4);
    \fill[dominoV] (2,0) rectangle (3,2);
    \fill[dominoV] (2,2) rectangle (3,4);
  \end{tikzpicture}
  \quad
  \begin{tikzpicture}[scale=0.5]
    \draw[step=1cm,gray] (0,0) grid (3,4);
    \fill[dominoV] (0,0) rectangle (1,2);
    \fill[dominoV] (0,2) rectangle (1,4);
    \fill[dominoH] (1,0) rectangle (3,1);
    \fill[dominoH] (1,1) rectangle (3,2);
    \fill[dominoH] (1,2) rectangle (3,3);
    \fill[dominoH] (1,3) rectangle (3,4);
  \end{tikzpicture}
  \quad
  \begin{tikzpicture}[scale=0.5]
    \draw[step=1cm,gray] (0,0) grid (3,4);
    \fill[dominoV] (0,0) rectangle (1,2);
    \fill[dominoH] (0,2) rectangle (2,3);
    \fill[dominoH] (0,3) rectangle (2,4);
    \fill[dominoH] (1,0) rectangle (3,1);
    \fill[dominoH] (1,1) rectangle (3,2);
    \fill[dominoV] (2,2) rectangle (3,4);
  \end{tikzpicture}
  \quad
  \begin{tikzpicture}[scale=0.5]
    \draw[step=1cm,gray] (0,0) grid (3,4);
    \fill[dominoH] (0,0) rectangle (2,1);
    \fill[dominoH] (0,1) rectangle (2,2);
    \fill[dominoV] (0,2) rectangle (1,4);
    \fill[dominoH] (1,2) rectangle (3,3);
    \fill[dominoH] (1,3) rectangle (3,4);
    \fill[dominoV] (2,0) rectangle (3,2);
  \end{tikzpicture}
  \quad
  \begin{tikzpicture}[scale=0.5]
    \draw[step=1cm,gray] (0,0) grid (3,4);
    \fill[dominoH] (0,0) rectangle (2,1);
    \fill[dominoH] (0,1) rectangle (2,2);
    \fill[dominoH] (0,2) rectangle (2,3);
    \fill[dominoH] (0,3) rectangle (2,4);
    \fill[dominoV] (2,0) rectangle (3,2);
    \fill[dominoV] (2,2) rectangle (3,4);
  \end{tikzpicture}

  \medskip

  \begin{tikzpicture}[scale=0.5]
    \draw[step=1cm,gray] (0,0) grid (3,4);
    \fill[dominoV] (0,0) rectangle (1,2);
    \fill[dominoV] (0,2) rectangle (1,4);
    \fill[dominoV] (1,0) rectangle (2,2);
    \fill[dominoH] (1,2) rectangle (3,3);
    \fill[dominoH] (1,3) rectangle (3,4);
    \fill[dominoV] (2,0) rectangle (3,2);
  \end{tikzpicture}
  \quad
  \begin{tikzpicture}[scale=0.5]
    \draw[step=1cm,gray] (0,0) grid (3,4);
    \fill[dominoV] (0,0) rectangle (1,2);
    \fill[dominoV] (0,2) rectangle (1,4);
    \fill[dominoH] (1,0) rectangle (3,1);
    \fill[dominoV] (1,1) rectangle (2,3);
    \fill[dominoH] (1,3) rectangle (3,4);
    \fill[dominoV] (2,1) rectangle (3,3);
  \end{tikzpicture}
  \quad
  \begin{tikzpicture}[scale=0.5]
    \draw[step=1cm,gray] (0,0) grid (3,4);
    \fill[dominoV] (0,0) rectangle (1,2);
    \fill[dominoV] (0,2) rectangle (1,4);
    \fill[dominoH] (1,0) rectangle (3,1);
    \fill[dominoH] (1,1) rectangle (3,2);
    \fill[dominoV] (1,2) rectangle (2,4);
    \fill[dominoV] (2,2) rectangle (3,4);
  \end{tikzpicture}
  \quad
  \begin{tikzpicture}[scale=0.5]
    \draw[step=1cm,gray] (0,0) grid (3,4);
    \fill[dominoV] (0,0) rectangle (1,2);
    \fill[dominoH] (0,2) rectangle (2,3);
    \fill[dominoH] (0,3) rectangle (2,4);
    \fill[dominoV] (1,0) rectangle (2,2);
    \fill[dominoV] (2,0) rectangle (3,2);
    \fill[dominoV] (2,2) rectangle (3,4);
  \end{tikzpicture}
  \quad
  \begin{tikzpicture}[scale=0.5]
    \draw[step=1cm,gray] (0,0) grid (3,4);
    \fill[dominoH] (0,0) rectangle (2,1);
    \fill[dominoV] (0,1) rectangle (1,3);
    \fill[dominoH] (0,3) rectangle (2,4);
    \fill[dominoV] (1,1) rectangle (2,3);
    \fill[dominoV] (2,0) rectangle (3,2);
    \fill[dominoV] (2,2) rectangle (3,4);
  \end{tikzpicture}
  \quad
  \begin{tikzpicture}[scale=0.5]
    \draw[step=1cm,gray] (0,0) grid (3,4);
    \fill[dominoH] (0,0) rectangle (2,1);
    \fill[dominoH] (0,1) rectangle (2,2);
    \fill[dominoV] (0,2) rectangle (1,4);
    \fill[dominoV] (1,2) rectangle (2,4);
    \fill[dominoV] (2,0) rectangle (3,2);
    \fill[dominoV] (2,2) rectangle (3,4);
  \end{tikzpicture}
  \caption{Domino tilings of a $4\times 3$ rectangle: $t(4,3;2)=a^6+6a^4b^2+4a^2b^4$.}
  \label{fig:11tilings}
\end{figure}
\begin{example}
  We have $t(4,3;2)=a^6+6a^4b^2+4a^2b^4$ as shown in Figure \ref{fig:11tilings}.
\end{example}
We are mainly interested in the following problem. 
\begin{problem}
  Determine the generating function $\sum_{n\geq 0}t(m,n;k)x^n$.
\end{problem}
Define
\begin{equation}\label{eq:Fk}
F_k(x)=F_k(x;a,b):=\sum_{n\geq 0}t(2k,n;k)x^n.  
\end{equation}
The following theorem is our main result.

\begin{theorem}\label{thm:main}
For each integer $k>1$
\[
F_k(x;a,b)=\frac{(1 - b^{2k}x^k)^{k-1}f_{k-1}(x;a,b^k)}
     {p(\sqrt{x})p(-\sqrt{x})(1 - b^{2k}x^k)^{k-1}f_{k-1}(x;a,-b^k)-(k-1)a^kb^kx^kf_{k-1}(x;a,b^k)},
\]
where $p(x)=1-ax-b^kx^k$ and
\begin{align}
f_N(x;a,b) := \sum_{s=0}^{\frac{N(N+1)}{2}} c_s(N)\, a^{{\rm rem}(2s,N+1)}\, b^{\left\lfloor\frac{2s}{N+1}  \right\rfloor} x^s.  \label{eq:fnxab}   
\end{align}
Here ${\rm rem}(2s,N+1)$ denotes the remainder when $2s$ is divided by $N+1$ and 
\begin{equation*}
c_s(N)=\begin{cases}  
1 & \text{if }s=0,\\
(-1)^{q_1+1} \dbinom{N - r_1 }{q_1} \dbinom{q_1+r_1-\left\lfloor \frac{N}{2} \right\rfloor - 1 }{q_1} & \text{if } \lfloor\frac{2s}{N+1}  \rfloor \mbox{ is odd} \text{ and } r_1 > \lfloor\frac{N}{2}\rfloor, \\[8pt]
(-1)^{q_0} \dbinom{  r_0-1 }{ q_0-1} \dbinom{\left\lfloor \frac{N}{2} \right\rfloor - r_0 + q_0}{q_0} & \text{if } \lfloor\frac{2s}{N+1}  \rfloor \mbox{ is even} \text{ and } r_0 \leq \lfloor\frac{N}{2}\rfloor, \\
0 & \text{otherwise,}
\end{cases} 
\end{equation*} 
where $q_j$ and $r_j$ are the quotient and remainder when $s$ is divided by $N+j$ for $j\in \{0,1\}$.
\end{theorem}

We also give a combinatorial expression for $f_N(x;a,b)$ as a signed sum over odd compositions (see Equation~\eqref{eq:fnweighted}). The combinatorial representation is useful for structural and tableau-theoretic arguments, but for large $N$ it is considerably less efficient for explicit computation than the explicit coefficient form in Equation~\eqref{eq:fnxab}.

\begin{example}
  We have
  \begin{align*}
    F_2(x;a,b)&=\frac{1 - b^4 x^2}{1 - a^2 x - \bigl(3a^2 b^2 + 2b^4\bigr)x^2 - a^2 b^4 x^3 + b^8 x^4}\\
    &=1+ a^{2}x+ \bigl(a^{4} + 3a^{2}b^{2} + b^{4}\bigr)x^{2}+ \bigl(a^{6} + 6a^{4}b^{2} + 4a^{2}b^{4}\bigr)x^{3}\\
&\qquad+ \bigl(a^{8} + 9a^{6}b^{2} + 16a^{4}b^{4} + 9a^{2}b^{6} + b^{8}\bigr)x^{4}+\cdots. 
  \end{align*}
  The coefficient of $x^3$ in $F_2(x;a,b)$ corresponds to the tilings shown in Figure~\ref{fig:11tilings}.
\end{example}
\begin{example}
 Setting $k=3$ in Theorem \ref{thm:main}, we obtain $F_3(x;a,b)=P/Q$, where
  \begin{align*}
P&=  1
  - a b^3 x^2
  - 3 b^6 x^3
  + 2 a b^9 x^5
  + 3 b^{12} x^6
  - a b^{15} x^8
   - b^{18} x^9,
  \end{align*}
  and 
  \begin{align*}
  Q&=1
  - a^2 x
  - a b^3 x^2
  - \bigl(3a^3 b^3 + 4b^6\bigr)x^3
  + a^2 b^6 x^4
  + \bigl(2a^4 b^6 + 3a b^9\bigr)x^5
  + \bigl(4a^3 b^9 + 6b^{12}\bigr)x^6\\
  &\quad+ a^2 b^{12} x^7
  - 3a b^{15} x^8
  - \bigl(a^3 b^{15} + 4b^{18}\bigr)x^9
  - a^2 b^{18} x^{10}
  + a b^{21} x^{11}
  + b^{24} x^{12}.
  \end{align*}
Therefore, the generating function for straight tromino tilings of a $6\times n$ rectangle is
  \begin{align*}
F_3(x;1,1) &=  \frac{
  1
  - x^2
  - 3x^3
  + 2x^5
  + 3x^6
  - x^8
  - x^9
}{
  1
  - x
  - x^2
  - 7x^3
  + x^4
  + 5x^5
  + 10x^6
  + x^7
  - 3x^8
  - 5x^9
  - x^{10}
  + x^{11}
  + x^{12}
     }\\[6pt]
    &=1 +  x +  x^{2} + 6 x^{3} + 13 x^{4} + 22 x^{5} + 64 x^{6} + 155 x^{7} + 321 x^{8} + 783 x^{9}+\cdots,
  \end{align*}
  corresponding to the sequence \href{https://oeis.org/A236577}{A236577} in the Online Encyclopedia of Integer Sequences (OEIS).
\end{example}

Currently, the OEIS contains the generating functions $F_k(x;1,1)$ in Theorem \ref{thm:main} for $2\leq k\leq 10$: 
\href{http://oeis.org/A005178}{A005178}, \href{http://oeis.org/A236577}{A236577}, \href{http://oeis.org/A236582}{A236582}, \href{http://oeis.org/A247117}{A247117}, \href{http://oeis.org/A250663}{A250663}, \href{http://oeis.org/A250664}{A250664}, \href{http://oeis.org/A250665}{A250665}, \href{http://oeis.org/A250666}{A250666}, and \href{http://oeis.org/A250667}{A250667}. Most of these generating functions appear to have been obtained by direct brute-force enumeration using the transfer-matrix method. Theorem \ref{thm:main} implies that the numerator of $F_k(x;a,b)$ is, a priori, a polynomial of degree $3{k \choose 2}$ in $x$ while the denominator has degree $3{k \choose 2}+k$, a pentagonal number.
Using Theorem \ref{thm:main}, these generating functions can be computed for much larger values of $k$, as exemplified by the following calculation, which took under 7 seconds on a personal computer.
\begin{example}
  The number of tilings of a $62\times 3141$ rectangle using $31\times 1$ tiles is given by
  \begin{align*}
    &13402557618801701815685444925429379056914684414597971761745488045182\\
    & 60506181047183540953319585706522974237498150073347365953548828874598\\
    & 60849814074167537144921607298786734849307555723438800870146833283846\\
    & 5955126575180559822761044422243837857742218930.
  \end{align*}
\end{example}
Given formal power series $A(x)=\sum_{n\geq 0}a_nx^n$ and $B(x)=\sum_{n\geq 0}b_nx^n$, their Hadamard product is the series
\[
A(x)*B(x):=\sum_{n\geq 0}a_nb_nx^n.
\]
One of the key ingredients in the proof of Theorem \ref{thm:main} is the following explicit result on Hadamard products.
  \begin{theorem}\label{thm:hadamard}
If $N\geq 2$ is a positive integer and $p_N(x)=1-ax-bx^N$, then
$$
\frac{1}{p_N(x)}*\frac{1}{p_N(x)}=\frac{f_{N-1}(x;a,b)}{p_N(\sqrt{x})p_N(-\sqrt{x})f_{N-1}(x;a,-b)},
$$
where $f_N(x;a,b)$ is defined by \eqref{eq:fnxab}.   
  \end{theorem}
One of the main contributions of this paper is a derivation of Theorem \ref{thm:hadamard} using enumerative results on special rim-hook tableaux (see Section \ref{subsec:srht}). These tableaux play an integral role in the representation theory of the symmetric group.
We rely on several external ingredients: Littlewood's identity on Schur expansions of plethystic compositions of elementary symmetric functions (Theorem~\ref{thm:threshold}).
On the tiling side, we use Graham's characterization of fault-free rectangle tilings (Theorem~\ref{thm:graham}).
We also invoke structural results of Aggarwal and Ram in the regime $m<2k$ (Proposition~\ref{prop:ARstructure}).

The paper is organized as follows.
Section~\ref{sec:central-horizontal-fault} decomposes $2k\times n$ tilings by vertical and central horizontal faults and derives the corresponding generating-function relations, including a Hadamard-product formulation.
Section~\ref{sec:refinement} derives explicit formulas for the polynomial $f_N(x;a,b)$ above.
Section~\ref{sec:hadamard-tilings} proves the Hadamard-product identity of Theorem~\ref{thm:hadamard} via special rim-hook tableaux and Littlewood's plethystic expansion identity.
Section~\ref{sec:final-hadamard-proof} combines this identity with Graham's fault-free criterion and Klarner's divisibility criterion to prove Theorem~\ref{thm:main}.

\section{Tilings with a central horizontal fault}\label{sec:central-horizontal-fault}
By convention, we place an $m\times n$ rectangle in the coordinate plane with its bottom-left corner at the origin.
For an integer $a$ with $0<a<n$, we say a tiling has a \emph{vertical fault at $x=a$} if the interior of no tile intersects the line $x=a$.
For an integer $b$ with $0<b<m$, we say a tiling has a \emph{horizontal fault at $y=b$} if the interior of no tile intersects the line $y=b$.
A tiling is \emph{vertically fault-free} if it has no vertical fault. In this section we isolate tilings of a $2k\times n$ rectangle that have a central horizontal fault and encode them by compositions. This yields a Hadamard-product formulation for the associated generating function.

Let $\mathcal{C}(2k,n;k)$ denote the collection of all tilings in $\mathcal{T}(2k,n;k)$ which have a central horizontal fault. An example of such a tiling is shown in Figure \ref{fig:vff}. Tilings $T\in \mathcal{T}(k,n;k)$ are in bijection with compositions of $n$ which have all parts in $\{1,k\}$. Therefore, each tiling $T\in \mathcal{C}(2k,n;k)$ corresponds to a pair of such compositions $u(T)$ and $d(T)$. In the running example of Figure \ref{fig:vff}, these compositions are $u(T)=(4,4,1,4,1,1,4)$ and $d(T)=(1,4,1,1,4,1,4,1,1,1)$. We refer to a contiguous collection of $k$ horizontal tiles placed one above the other as a \emph{block}. For a tiling $T\in \mathcal{C}(2k,n;k)$, the blocks correspond to the parts equal to $k$ in the compositions $u(T)$ and $d(T)$. If we assign a weight of $b$ to each horizontal tile, then the tiling in Figure \ref{fig:vff} has weight $b^{28}$.
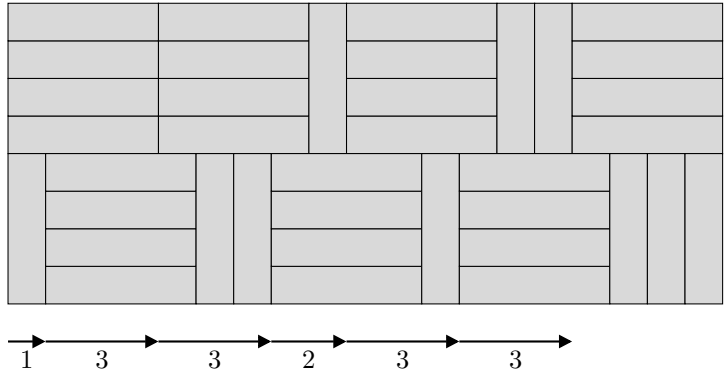
\begin{figure}[!ht]
  \centering
\begin{tikzpicture}[scale=0.5] 
  \tikzset{tile/.style={fill=gray!30,draw=black}}
  \newcommand{\hTile}[2]{\fill[tile] (#1,#2) rectangle ++(4,1);}
  \newcommand{\vTile}[2]{\fill[tile] (#1,#2) rectangle ++(1,4);}
  \foreach \y in {4,5,6,7}{ \hTile{0}{\y} }
  \foreach \y in {4,5,6,7}{ \hTile{4}{\y} }
  \vTile{8}{4}
  \foreach \y in {4,5,6,7}{ \hTile{9}{\y} }
  \vTile{13}{4} \vTile{14}{4}
  \foreach \y in {4,5,6,7}{ \hTile{15}{\y} }

  \vTile{0}{0}
  \foreach \y in {0,1,2,3}{ \hTile{1}{\y} }
  \vTile{5}{0} \vTile{6}{0}
  \foreach \y in {0,1,2,3}{ \hTile{7}{\y} }
  \vTile{11}{0}
  \foreach \y in {0,1,2,3}{ \hTile{12}{\y} }
  \vTile{16}{0} \vTile{17}{0} \vTile{18}{0}

  \begin{scope}[shift={(0,-1)}]
  \draw[->, line width=0.8pt, >={Triangle[length=5pt,width=5pt]}]
    (0,0)  -- (1,0)  node[midway,below] {1};
  \draw[->, line width=0.8pt, >={Triangle[length=5pt,width=5pt]}]
    (1,0)  -- (4,0)  node[midway,below] {3};
  \draw[->, line width=0.8pt, >={Triangle[length=5pt,width=5pt]}]
    (4,0)  -- (7,0)  node[midway,below] {3};
  \draw[->, line width=0.8pt, >={Triangle[length=5pt,width=5pt]}]
    (7,0)  -- (9,0)  node[midway,below] {2};
  \draw[->, line width=0.8pt, >={Triangle[length=5pt,width=5pt]}]
    (9,0)  -- (12,0) node[midway,below] {3};
  \draw[->, line width=0.8pt, >={Triangle[length=5pt,width=5pt]}]
    (12,0) -- (15,0) node[midway,below] {3};
\end{scope}
\end{tikzpicture}
   \caption{A tiling $T\in \mathcal{C}(8,19;4)$ that is vertically fault-free; the arrows indicate the successive offsets of the blocks.} 
  \label{fig:vff} 
\end{figure}

Define
\begin{align*}
  h_n=h_{k,n}(a,b):=\sum_{T}w(T),
\end{align*}
where the sum is taken over all tilings $T\in \mathcal{T}(2k,n;k)$ which have a horizontal fault at $y=k$. Similarly define
\begin{align*}
  v_n=v_{k,n}(a,b):=\sum_{T}w(T),
\end{align*}
where the sum is taken over all vertically fault-free tilings $T\in \mathcal{T}(2k,n;k)$ which have a horizontal fault at $y=k$. Note that $h_{k,n}(a,b)$ and $v_{k,n}(a,b)$ are homogeneous polynomials in $a$ and $b$ of degree $2n$. Define
\begin{equation}
  \label{eq:hkvk}
  H_k(x)=\sum_{n\geq 0}h_{k,n}(a,b)x^n \mbox{ and } V_k(x)=V_k(x;a,b)=\sum_{n\geq 1}v_{k,n}(a,b)x^n.
\end{equation}

\begin{proposition}\label{prop:Hkvkandhadamard}  
If $p_k(x;a,b)=1-ax-bx^k$, then 
\begin{equation*}
H_k(x)=\frac{1}{1-V_k(x)}=\frac{1}{p_k(x;a,b^k)}*\frac{1}{p_k(x;a,b^k)}, 
\end{equation*}
\end{proposition}
\begin{proof}
  Since each tiling corresponds to a unique sequence of vertically fault-free tilings, one obtains the first equality. The second equality follows since $k$-bar tilings of a $k\times n$ rectangle are in bijection with compositions of $n$ with parts in $\{1,k\}$. 
\end{proof}
\begin{remark}\label{rem:vkhomogeneity}
 By the homogeneity of $v_{k,n}(a,b)$, the power series $V_k(x;a,b)$ can be recovered from its value at $a=1$. More precisely, $V_k(x;a,b)=V_k(a^2x;1,ba^{-1})$, so it suffices to compute the generating function $V_k(x;1,b)$.  
\end{remark}

Within $\mathcal{C}(2k,n;k)$, let $VC(2k,n;k)$ denote the subset of vertically fault-free tilings. Note that the positions of the blocks completely determine whether such a tiling is vertically fault-free. 
\begin{definition}
  Let ${\rm TComp}(N,k)$ denote the set of all compositions $\beta$ of $N$ with part sizes at most $k-1$ such that each pair of consecutive parts of $\beta$ has sum at least $k$.
\end{definition}
\begin{example}
We have ${\rm TComp}(6,3)=\{(2,2,2),(1,2,1,2),(1,2,2,1),(2,1,2,1)\}.$
\end{example}
\begin{proposition}
  For $n>k\geq 2$, we have $ |VC(2k,n;k)|=2\cdot |{\rm TComp}(n-k,k)|.$
\end{proposition}
\begin{proof}
  For $n>k$, let $T\in VC(2k,n;k)$ be a tiling such that the first part of $u(T)$ is $k$. Scanning $T$ from left to right, we see that the vertically fault-free condition implies that the blocks alternate in the upper and lower halves of $T$. Consider the composition $\beta(T)=(\beta_1,\ldots,\beta_\ell)$ that specifies the relative offsets of the blocks as one scans $T$ from left to right. In the example of Figure \ref{fig:vff}, we have $\beta(T)=(1,3,3,2,3,3)$. For $T\in VC(2k,n;k)$, the composition $\beta(T)$ has size $n-k$ and each part of $\beta(T)$ is less than $k$. Further, we must have $\beta_i+\beta_{i+1}\geq k$ for $1\leq i\leq \ell(\beta)-1$ since successive blocks on the same side of the central fault are offset by at least $k$. Conversely, it is clear that any composition $\beta\in {\rm TComp}(n-k,k)$ corresponds to a vertically fault-free tiling $T\in VC(2k,n;k)$ such that $u(T)$ has first part $k$.

  If $T\in VC(2k,n;k)$ and  $u(T)$ has first part 1, then reflecting $T$ about a horizontal axis yields a tiling $T'\in VC(2k,n;k)$ such that the first part of $u(T')$ is $k$. This accounts for the factor of 2 in the statement of the proposition.
\end{proof}
Since each tiling $T\in VC(2k,n;k)$ has one block more than the number of parts of $\beta(T)$, we obtain the following result.
\begin{corollary}\label{cor:vnrec}
  For $n>k\geq 2,$
  \begin{align*}
    v_{k,n}(1,b)=2 b^k\sum_{\beta\in {\rm TComp}(n-k,k)}b^{k\cdot \ell(\beta)}.
  \end{align*}
\end{corollary}
Define
\begin{align*}
  c(n)=\sum_{\beta\in {\rm TComp}(n,k)}{b}^{k\cdot \ell(\beta)},
\end{align*}
and let $C(x)=\sum_{n\geq 0}c(n)x^n$.
\begin{proposition}\label{prop:VfromC} 
  We have
  \begin{align*}
    V_k(x;1,b)=x+b^{2k}x^k+2b^kx^kC(x).
  \end{align*}
\end{proposition}
\begin{proof}
     Since $v_{k,n}(a,b)=0$ for $1<n<k$ and precisely 3 tilings contribute to $v_{k,k}(a,b)$, we obtain $V_k(x;a,b)=a^2x+(2a^kb^k+b^{2k})x^k+\cdots.$ Together with Corollary \ref{cor:vnrec} we obtain
  \begin{align*}
    V_k(x;1,b)=x+(2b^k+b^{2k})x^k+2b^kx^k\sum_{n>k}c(n-k)x^{n-k},
  \end{align*}
  and the result follows since the sum above equals $C(x)-1$.
\end{proof}
\begin{proposition}\label{prop:vk}
  For each positive integer $k\geq 2$, there exists a polynomial $\widetilde{P}_k(x)$ such that
  \begin{align*}
C(x)=\frac{\widetilde{P}_k(x)}{\det(I-A(x))},
  \end{align*}
  where $A(x)=(a_{ij})$ is the $(k-1)\times (k-1)$ matrix with entries given by
  \begin{equation}\label{eq:aij}
  a_{ij}=
  \begin{cases}
    b^k x^j & \mbox{if }i+j\geq k,\\ 
    0 & \mbox{otherwise.}
  \end{cases}
\end{equation}
\end{proposition}
\begin{proof}
  We use a transfer matrix approach to derive the generating function for $c(n)$. Consider the graph $G$ with vertex set $V(G)=\{1,2,\ldots,k-1\}$ and edge set $E(G)=\{(i,j):i+j\geq k\}$. Assign a weight of $a_{ij}$ to each edge $(i,j)$. 
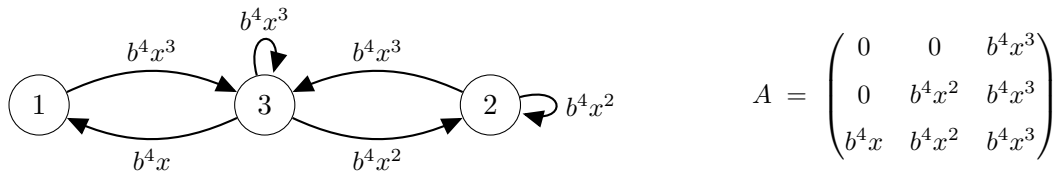
\begin{figure}[ht]
  \centering
  \tikzset{
    vertex/.style={
      circle, draw, fill=white,
      minimum size=8mm, inner sep=1pt, font=\large
    },
    myarrow/.style={
      -{triangle 45}, thick
    }
  }

  \begin{minipage}[c]{0.6\linewidth}
    \centering
    \begin{tikzpicture}[auto]
      \node[vertex] (v1) at (0,0)   {1};
      \node[vertex] (v3) at (3,0)   {3};
      \node[vertex] (v2) at (6,0)   {2};

      \draw[myarrow, loop right, looseness=8, shorten >=1pt]
        (v2) to node[right] {$b^4x^{2}$} (v2);
      \draw[myarrow, loop above, looseness=8, shorten >=1pt]
        (v3) to node[above] {$b^4x^{3}$} (v3);

      \draw[myarrow, bend left=25]
        (v1) to node[midway, above, sloped] {$b^4x^{3}$} (v3);
      \draw[myarrow, bend left=25]
        (v3) to node[midway, below, sloped] {$b^4x$}     (v1);

      \draw[myarrow, bend right=25]
        (v2) to node[midway, above, sloped] {$b^4x^{3}$} (v3);
      \draw[myarrow, bend right=25]
        (v3) to node[midway, below, sloped] {$b^4x^{2}$} (v2);
    \end{tikzpicture}
  \end{minipage}
  \hfill
  \begin{minipage}[c]{0.35\linewidth}
    \[
      A \;=\;
      \begin{pmatrix}
        0      & 0       & b^4x^{3}\\[6pt]
        0      & b^4x^{2}& b^4x^{3}\\[6pt]
        b^4x   & b^4x^{2}& b^4x^{3}
      \end{pmatrix}
    \]
  \end{minipage}

  \caption{The graph $G$ and its weighted adjacency matrix \(A\) for $k=4$.}
  \label{fig:linear-graph-k4}
\end{figure}

Then $A=A(x)$ is the $(k-1)\times (k-1)$ weighted adjacency matrix of $G$. When $k=4$, the graph $G$ is shown in Figure \ref{fig:linear-graph-k4}. Define the row vectors $v=(b^k x,b^k x^2,\ldots,b^k x^{k-1})$ and $u=(1,\ldots,1)$ of length $k-1$. Then the generating function $C(x)=\sum_{n\geq 0}c(n)x^n$ is given by
\begin{align*} 
  C(x)=1+v(I-A(x))^{-1}u^T=1+vu^T+vA(x)(I-A(x))^{-1}u^T,
\end{align*}
where $u^T$ denotes the transpose of $u$. It follows that 
\begin{align*}
  C(x)=\frac{\widetilde{P}_k(x)}{Q_k(x)},\qquad 
  Q_k(x)=\det(I-A(x)),
\end{align*}
with
\begin{align*}
  P_k(x)&:=vA(x)\,{\rm adj}(I-A(x))u^T,\qquad
  \widetilde{P}_k(x):=(1+vu^T)Q_k(x)+P_k(x).\qedhere
\end{align*}
\end{proof}
\subsection{Coprimeness of $P_k(x)$ and $Q_k(x)$}
We now work in the coefficient ring $K[x]$ where $K = \CC(b)$. Our immediate objective is to show that if $V=V_k(x;1,b)$ is written as a quotient $V=V_1(x)/V_2(x)$ of coprime polynomials in $K[x]$, then $V_2(x)=\det(I-A(x))$.  Recall that
$$
P_k(x) = v A(x) \, \mathrm{adj}(I - A(x)) \, u^T,
$$
and $Q_k(x)=\det(I-A(x))$. From the identity for $C(x)$ obtained in the proof of Proposition \ref{prop:vk}, the numerator of $C(x)$ is
\[
\widetilde{P}_k(x)=(1+vu^T)Q_k(x)+P_k(x).
\]
Hence it is enough to prove that $P_k(x)$ and $Q_k(x)$ are coprime. We decompose the matrix $A(x)$ into $A(x) = b^k T D(x)$, where $D(x) = \mathrm{diag}(x, x^2, \dots, x^{k-1})$ and $T=(t_{ij})$ is the $(k-1) \times (k-1)$ symmetric threshold matrix with $t_{ij} = 1$ if $i+j \ge k$, and $0$ otherwise. Notice that $v$ can be rewritten as $b^k \mathbf{1}^T D(x)$, and $u^T = \mathbf{1}$ (where $\mathbf{1}$ is the all-ones column vector). Thus, the auxiliary numerator evaluates to
\begin{equation} \label{eq:Pk_exact}
P_k(x) = b^{2k} \mathbf{1}^T D(x) T D(x) \, \mathrm{adj}(I - b^k T D(x)) \, \mathbf{1}.
\end{equation}

We require a result from linear algebra regarding the adjugate of a matrix.
\begin{lemma} \label{lem:adjugate}
Let $B$ be an $n \times n$ matrix over a field. If $\operatorname{rank}(B) = n-1$, then $\operatorname{rank}(\mathrm{adj}(B)) = 1$. Furthermore, there exists a nonzero scalar $c$ such that $\mathrm{adj}(B) = c \, w z^T$, where $w$ and $z^T$ are nonzero right and left nullvectors of $B$, respectively.
\end{lemma}
\begin{proof}
Since $\operatorname{rank}(B) = n-1$, the adjugate $\mathrm{adj}(B)$ is nonzero. From the identity $B \, \mathrm{adj}(B) = \mathrm{adj}(B) B = \det(B)I = 0$, the column space of $\mathrm{adj}(B)$ is contained in the right nullspace of $B$, and its row space is contained in the left nullspace of $B$. Since both nullspaces are one-dimensional, $\mathrm{adj}(B)$ must have rank 1 and can be factored as $c \, w z^T$ as claimed. 
\end{proof}

\begin{theorem}
For any integer $k \ge 2$, the polynomials $P_k(x)$ and $Q_k(x)$ are coprime in $K[x]$, where $K = \CC(b)$.
\end{theorem}

\begin{proof}
Assume, for the sake of contradiction, that $\gcd(P_k, Q_k) \neq 1$ in $K[x]$. Then there exists a root $\lambda \in \overline{K}$ (the algebraic closure of $K$) such that $P_k(\lambda) = Q_k(\lambda) = 0$. Since $Q_k(0) = \det(I) = 1$ in $K$, we know that $\lambda \neq 0$. 

The condition $Q_k(\lambda) = 0$ implies that the matrix $I - b^k T D(\lambda)$ is singular. Let $L = K(\lambda) = \CC(b)(\lambda)$. Choose a vector $w \in L^{k-1}$ such that $(I - b^k T D(\lambda)) w = 0$. This implies $w = b^k T D(\lambda) w$. Writing this relation out row-by-row gives the linear system
\begin{equation} \label{eq:recurrence}
w_i = b^k \sum_{j= k-i}^{k-1} \lambda^j w_j \quad \text{for} \quad 1 \le i \le k-1
\end{equation}
By subtracting adjacent rows, we obtain the difference equation $w_i - w_{i-1} = b^k \lambda^{k-i} w_{k-i}$ for $i \ge 2$, with the base case $w_1 = b^k \lambda^{k-1} w_{k-1}$ at $i=1$. 

We define the linear functional $\phi: \ker(I - b^k T D(\lambda)) \to L$ by $\phi(w) = w_{k-1}$. We claim that $\phi$ is injective.

\textit{Proof of claim:}
Suppose $\phi(w) = w_{k-1} = 0$. We will prove by induction on $m$ that $w_m = 0$ and $w_{k-m} = 0$ for all $1 \le m \le k-1$. Consider the base case $m=1$. We are given $w_{k-1} = 0$. The system \eqref{eq:recurrence} at $i=1$ yields $w_1 = b^k \lambda^{k-1} (0) = 0$.
Now assume $w_{k-j} = 0$ and $w_j = 0$ for all $1 \le j < m$. We evaluate the difference equation at $i = k-m+1$:
    \[
    w_{k-m+1} - w_{k-m} = b^k \lambda^{m-1} w_{m-1}
    \]
    By the inductive hypothesis, $w_{k-m+1} = 0$ and $w_{m-1} = 0$ which implies $w_{k-m} = 0$. Next, evaluate the difference equation at $i = m$:
    \[
    w_m - w_{m-1} = b^k \lambda^{k-m} w_{k-m}
    \]
    Since $w_{m-1} = 0$ and we just established $w_{k-m} = 0$, we have $w_m = 0$. By induction, every component of $w$ is zero, proving the claim.

Because $\phi$ maps the kernel injectively into the one-dimensional space $L$, we have $\dim \ker(I - b^k T D(\lambda)) \le 1$. Since the kernel is nontrivial, $\dim \ker(I - b^k T D(\lambda)) = 1$. 
Let $w$ be a nonzero right nullvector for $I - b^k T D(\lambda)$. Define $z = D(\lambda) w$ and note that $z$ is nonzero since $\lambda \neq 0$. We verify $z^T$ is a left nullvector for $I - b^k T D(\lambda)$:
\begin{equation*}
(I - b^k T D(\lambda))^T z = (I - b^k D(\lambda) T) D(\lambda) w = D(\lambda) (I - b^k T D(\lambda)) w = 0.
\end{equation*}
By Lemma \ref{lem:adjugate}, $\mathrm{adj}(I - b^k T D(\lambda)) = c \, w z^T$ for some nonzero scalar $c \in L$. Substituting this into Equation \eqref{eq:Pk_exact} at $x = \lambda$ yields:
\begin{align*}
P_k(\lambda) &= b^{2k} \mathbf{1}^T D(\lambda) T D(\lambda) (c \, w z^T) \mathbf{1} \\
&= c \, b^{2k} \left( \mathbf{1}^T D(\lambda) T D(\lambda) w \right) \left( z^T \mathbf{1} \right)
\end{align*}
Since $b^k T D(\lambda) w = w$, it follows that 
$$
P_k(\lambda) = c \, b^k (\mathbf{1}^T z)^2.
$$

For $P_k(\lambda)$ to equal $0$, we must have $\mathbf{1}^T z = 0$ (since $c \neq 0$). Evaluating the sum from Equation \eqref{eq:recurrence} precisely at $i = k-1$ gives:
$$
w_{k-1} = b^k \sum_{j=1}^{k-1} \lambda^j w_j = b^k \sum_{j=1}^{k-1} z_j = b^k (\mathbf{1}^T z)
$$
However, if $P_k(\lambda) = 0$, then $\mathbf{1}^T z = 0$, which implies $w_{k-1} = 0$. As proven earlier, $w_{k-1} = 0$ implies $w = 0$, which contradicts the fact that $w$ is a nonzero nullvector. Therefore, $P_k(\lambda)$ cannot be zero, and the polynomials $P_k(x)$ and $Q_k(x)$ are coprime in $K[x]$.
\end{proof}

\begin{corollary}
The generating function $V_k(x;1,b)$ can be expressed as a quotient of coprime polynomials in $\CC(b)[x]$ with denominator $Q_k(x)=\det(I-A(x))$. 
\end{corollary}
\begin{proof}
Let $d(x)$ be a common divisor of $\widetilde{P}_k(x)$ and $Q_k(x)$. Since
\[
\widetilde{P}_k(x)-(1+vu^T)Q_k(x)=P_k(x),
\]
the polynomial $d(x)$ also divides $P_k(x)$. By the theorem, $\gcd(P_k,Q_k)=1$, so $d(x)=1$. Hence $\gcd(\widetilde{P}_k,Q_k)=1$.

Now Proposition \ref{prop:VfromC} gives
\[
V_k(x;1,b)=x+b^{2k}x^k+2b^kx^k\frac{\widetilde{P}_k(x)}{Q_k(x)}
=\frac{(x+b^{2k}x^k)Q_k(x)+2b^kx^k\widetilde{P}_k(x)}{Q_k(x)}.
\]
If a polynomial divides both $Q_k(x)$ and the numerator above, then it divides
\(2b^kx^k\widetilde{P}_k(x)\), hence it divides $x^k\widetilde{P}_k(x)$. Since $Q_k(0)=1$, such a divisor cannot have $x$ as a factor. Therefore it must divide $\widetilde{P}_k(x)$, and thus it is $1$. This proves the claim.
\end{proof}
\subsection{A multivariate determinant}    
We do not know a direct method to evaluate $\det(I-A(x))$. We proceed by considering a multivariate generalization, which specializes to the desired determinant. Consider the $r\times r$ determinant $$D_r(\xx)=D_r(x_1,\ldots,x_r):=\det (I-M_r),$$
where  $M_r=(m_{ij})$ denotes the $r\times r$ matrix with entries
\begin{equation}\label{eq:mat}
  m_{ij}=
  \begin{cases}
     x_j & \mbox{if }i+j\geq r+1,\\ 
    0 & \mbox{otherwise.}
  \end{cases}
\end{equation}
For example, we have
\begin{align*}
  D_3(\xx) &= \det \begin{pmatrix}
1 & 0 & - x_3 \\
0 & 1- x_2 & - x_3 \\
- x_1 & - x_2 & 1- x_3
  \end{pmatrix}=1-x_2-x_3-x_1x_3+x_1x_2x_3.
\end{align*}
Note that
\begin{equation}
  \label{eq:determ}
  \det(I-A(x))=D_{k-1}(b^k x,b^k x^2,\ldots,b^k x^{k-1}).
\end{equation}
For each positive integer $r$, let $\oc_{\leq r}$ denote the set of all odd compositions of integers not exceeding $r$ (including zero).
\begin{example}
  $\oc_{\leq 4}=\{\emptyset,(1),(1,1),(1,1,1),(3),(1,1,1,1),(3,1),(1,3)\}$.
\end{example}
\begin{remark}
  If we set $b_r=|\oc_{\leq r}|$, then by conditioning on the size of the first part of the odd composition, one obtains the recurrence $b_{r+2}=b_{r+1}+b_r$ for $r\geq1$ with $b_1=2$ and $b_2=3$. It follows that $b_r=F_{r+2}$, a Fibonacci number. 
\end{remark}
\begin{definition}
To each positive integer $r$, associate a canonical permutation $\sigma=\sigma_r$ defined by
\begin{align*}
  \sigma(i)=
  \begin{cases}
    \frac{i}{2} & i \mbox{ even}\\
    r+\frac{1-i}{2} & i \mbox{ odd}.
  \end{cases}
\end{align*}
\end{definition}

\begin{example}
For $r=7$, we have $  \sigma_r= 7162534$.
\end{example}
\begin{definition}\label{def:sr}
  Given a composition $\alpha\in \oc_{\leq r}$, define
  \begin{align*}
    S_r(\alpha)=\sigma_r(\alpha_1)+\sigma_r(\alpha_1+\alpha_2)+\cdots.
  \end{align*}
\end{definition}
\begin{example}
  For $r=6$ we have $\sigma_r=615243$. If $\alpha=(1,3,1)$, then $S_r(\alpha)=\sigma(1)+\sigma(4)+\sigma(5)=6+2+4=12$.
\end{example}
\begin{definition}\label{def:varphir}
For each positive integer $r$, define
\begin{equation*}
 \varphi_r(\xx)= \varphi_r(x_1,\ldots,x_r):=\sum_{\alpha\in \oc_{\leq r}}(-1)^{\left\lfloor \frac{\ell(\alpha)+1}{2}\right\rfloor}x_\alpha,
\end{equation*}
where   $\ell(\alpha)$ denotes the number of parts of $\alpha$ and  $x_\alpha:=\prod_{i=1}^{\ell(\alpha)}x_{\sigma_r(\alpha_1+\cdots+\alpha_i)}$.  
\end{definition}
\begin{example}
  We have $ \varphi_3(x_1,x_2,x_3)=1- x_2-x_3-x_1x_3+x_1x_2x_3.$
\end{example}
When the positive integer $r$ is clear from the context, we simply write $\sigma$ for $\sigma_r$. We will prove that $D_r(\xx)=\varphi_r(\xx)$ by showing that they both satisfy the same recursion.
\begin{proposition}\label{prop:frec}
  We have $\varphi_1(\xx)=1-x_1,\varphi_2(\xx)=1 -  x_2 - x_1 x_2$ and for $r\geq 3$,
  \[\varphi_r(x_1, \ldots , x_r) = \left (1 - x_1x_r + x_r \frac{\partial}{\partial x_{r-1}} \right ) \varphi_{r-2}(x_2, \ldots , x_{r-1}). \] 
\end{proposition}
\begin{proof}
The initial values $\varphi_1({\bf x})=1-x_1$ and $\varphi_2({\bf x})=1-x_2-x_1x_2$ follow directly from Definition~\ref{def:varphir}. Suppose $r\geq 3$ and write $\oc_{\leq r}=C_1\sqcup C_2\sqcup C_3$, where
\begin{align*}
  C_1&:=\{\alpha\in \oc_{\leq r}:\ell(\alpha)\geq 2,\ \alpha_1=\alpha_2=1\},\\
  C_2&:=\{\alpha\in \oc_{\leq r}:\alpha_1=1\text{ and }(\ell(\alpha)=1\text{ or }\alpha_2\geq 3)\},\\
  C_3&:=\{\emptyset\}\cup\{\alpha\in \oc_{\leq r}:\alpha_1\geq 3\}.
\end{align*}
Let $\sigma=\sigma_r$ and $\hat{\sigma}=\sigma_{r-2}$. Then $\sigma(1)=r$, $\sigma(2)=1$, and for $1\leq t\leq r-2$,
\[
\sigma(t+2)=\hat{\sigma}(t)+1.
\]
For $C_1$, every element has the form $\alpha=(1,1,\beta)$ with $\beta\in \oc_{\leq r-2}$. Hence
\begin{align*}
\sum_{\alpha\in C_1}(-1)^{\lfloor(\ell(\alpha)+1)/2\rfloor}x_\alpha
&=-x_1x_r\sum_{\beta\in \oc_{\leq r-2}}(-1)^{\lfloor(\ell(\beta)+1)/2\rfloor}
\prod_{i=1}^{\ell(\beta)}x_{\hat{\sigma}(\beta_1+\cdots+\beta_i)+1}\\
&=-x_1x_r\,\varphi_{r-2}(x_2,\ldots,x_{r-1}).
\end{align*}
For $C_3$, define a bijection $\Psi_3:C_3\to \oc_{\leq r-2}$ by
\[
\Psi_3(\emptyset)=\emptyset,\qquad
\Psi_3(\alpha_1,\alpha_2,\ldots)=(\alpha_1-2,\alpha_2,\ldots)\ \text{if }\alpha_1\geq 3.
\]
Using $\sigma(t+2)=\hat{\sigma}(t)+1$, the monomials and signs are preserved under $\Psi_3$, so
\[
\sum_{\alpha\in C_3}(-1)^{\lfloor(\ell(\alpha)+1)/2\rfloor}x_\alpha
=\varphi_{r-2}(x_2,\ldots,x_{r-1}).
\]

For $C_2$, first note that a monomial $x_\alpha$ contains the factor $x_r$ if and only if $\alpha_1=1$ (because $\sigma(1)=r$). Therefore
\[
\frac{\partial}{\partial x_r}\varphi_r(x_1,\ldots,x_r)
=\sum_{\substack{\alpha\in \oc_{\leq r}\\ \alpha_1=1}}
(-1)^{\lfloor(\ell(\alpha)+1)/2\rfloor}
\prod_{i=2}^{\ell(\alpha)}x_{\sigma(\alpha_1+\cdots+\alpha_i)}.
\]
Now define a bijection
\[
\Psi_2:C_2\longrightarrow\{\beta\in \oc_{\leq r-2}:\beta_1=1\}
\]
by
\[
\Psi_2(1)=(1),\qquad
\Psi_2(1,\alpha_2,\alpha_3,\ldots)=(1,\alpha_2-2,\alpha_3,\ldots)\ \text{if }\alpha_2\geq 3.
\]
Again using $\sigma(t+2)=\hat{\sigma}(t)+1$, we obtain
\[
\sum_{\alpha\in C_2}(-1)^{\lfloor(\ell(\alpha)+1)/2\rfloor}x_\alpha
=x_r\frac{\partial}{\partial x_{r-1}}\varphi_{r-2}(x_2,\ldots,x_{r-1}).
\]
Adding the three contributions gives
\begin{align*}
\varphi_r(x_1,\ldots,x_r)
&=\left(1-x_1x_r+x_r\frac{\partial}{\partial x_{r-1}}\right)\varphi_{r-2}(x_2,\ldots,x_{r-1}).\qedhere
\end{align*}
\end{proof}
We now consider the determinant $D_r(\xx)$. Given a matrix $A$, let $A^{j_1,\ldots,j_k}_{i_1,\ldots,i_k}$ denote the submatrix of $A$ obtained by omitting the rows with indices $i_1,\ldots,i_k$ and the columns with indices $j_1,\ldots,j_k$. The following identity (see Aigner \cite[Lem. 7.7]{MR2339282} or Krattenthaler \cite[Prop. 10]{MR1701596}) is the key to obtaining a recursion satisfied by $D_r(\xx)$.
\begin{lemma}[Jacobi-Desnanot identity]
  If $A$ is an $r\times r$ matrix with $r\geq 2$, then
  \begin{align*}
    \det A \cdot \det A^{1,r}_{1,r}=\det A^{1}_{1}\cdot  \det A^{r}_{r}-\det A^{r}_{1}\cdot \det A^{1}_{r}.
  \end{align*}
\end{lemma}
\begin{theorem}
  We have $D_r(\xx)=\varphi_r(\xx)$ for each positive integer $r$.
\end{theorem}
\begin{proof}
Recall that $D_r(\xx)=\det (I-M_r)$, where $M_r$ is defined in Equation \eqref{eq:mat}. We have the initial values $D_1(\xx)=1-x_1=\varphi_1(\xx)$ and $D_2(\xx)=1 - x_2 - x_1 x_2=\varphi_2(\xx)$. Now suppose $r\geq 3$ and let
\[
A=I-M_r=\begin{pmatrix}
1 & 0 & \cdots & 0 & -x_r\\[6pt]
0 & 1 & \cdots & -x_{r-1} & -x_r\\[6pt]
\vdots & \vdots & \ddots & \vdots & \vdots\\[6pt]
0 & -x_2 & \cdots & 1-x_{r-1} & -x_r\\[6pt]
-x_1 & -x_2 & \cdots & -x_{r-1} & 1-x_r
\end{pmatrix}.
\]
Observe that $\det A_{r}^r=\det A_{1,r}^{1,r}=D_{r-2}(x_2,\ldots,x_{r-1})$. Applying the Jacobi-Desnanot identity to $A$, we obtain
\begin{align*}
\det A\cdot \det A_{1,r}^{1,r}=\det A^{1}_{1}\cdot  \det A^{r}_{r}-x_1x_r \left(\det A^{1,r}_{1,r}\right)^2.
\end{align*}
It follows that
\begin{equation}\label{eq:1}
  \det A=\det A_{1}^1 -x_1x_r\det A_{1,r}^{1,r}.
\end{equation}
Now consider the submatrix $A_1^{1}$. Subtract the penultimate row from the last row to obtain
\begin{align*}
  \begin{vmatrix}
1       & 0      & \cdots & -x_{r-1}   & -x_r\\[6pt]
0       & 1  & \cdots & -x_{r-1}   & -x_r\\[6pt]
\vdots  & \vdots & \ddots & \vdots & \vdots\\[6pt]
-x_2   & -x_3   & \cdots & 1-x_{r-1}  & -x_r\\[6pt]
-x_2   & -x_3   & \cdots & -x_{r-1}   & 1-x_r
  \end{vmatrix}=
    \begin{vmatrix}
1       & 0      & \cdots & -x_{r-1}   & -x_r\\[6pt]
0       & 1  & \cdots & -x_{r-1}   & -x_r\\[6pt]
\vdots  & \vdots & \ddots & \vdots & \vdots\\[6pt]
-x_2   & -x_3   & \cdots & 1-x_{r-1}  & -x_r\\[6pt]
0   & 0   & \cdots & -1   & 1
\end{vmatrix}.
\end{align*}
Now expand the determinant by the last row. The cofactor of 1 is clearly $\det A_{1,r}^{1,r}$. On the other hand, one sees that the cofactor of $-1$ equals
\begin{align*}
  -x_r\frac{\partial}{\partial x_{r-1}}\det A_{1,r}^{1,r}.
\end{align*}
It follows that
\begin{equation}\label{eq:2}
  \det A_{1}^1=\det A_{1,r}^{1,r}+  x_r\frac{\partial}{\partial x_{r-1}}\det A_{1,r}^{1,r}.
\end{equation}
From Equations \eqref{eq:1} and \eqref{eq:2}, it follows that
\begin{align*}
  \det A=\left(1-x_1x_r+x_r\frac{\partial}{\partial x_{r-1}}\right)\det A_{1,r}^{1,r}, 
\end{align*}
which is equivalent to
\begin{align*}
  D_r(x_1,\ldots,x_r)=\left(1-x_1x_r+x_r\frac{\partial}{\partial x_{r-1}}\right)D_{r-2}(x_2,\ldots,x_{r-1}).
\end{align*}
By Proposition \ref{prop:frec} it follows that $D_r(\xx)=\varphi_r(\xx)$ for each $r\geq 1$.
\end{proof}
Define
\begin{equation}\label{eq:fnweighted}
f_N(x;a,b) := \sum_{\alpha\in \oc_{\leq N}}(-1)^{\left\lfloor \frac{\ell(\alpha)+1}{2}\right\rfloor}a^{{\rm rem}(\alpha)}b^{\ell(\alpha)}x^{S_{N}(\alpha)},
\end{equation}
where $S_N(\alpha)$ is given by Definition~\ref{def:sr} and ${\rm rem}(\alpha)$ denotes the remainder when $2S_N(\alpha)$ is divided by $N+1$. Note that $f_N(x;1,b)=\varphi_N(bx,bx^2,\ldots,bx^N)$ by Definition \ref{def:varphir}.
\begin{theorem}\label{thm:vkden}
  Let $V_k(x;1,b)$ denote the generating function for vertically fault-free tilings of a $2k\times n$ rectangle with $k\times 1$ tiles that have a central horizontal fault. Then $V_k(x;1,b)=P(x)/Q(x)$ where
    \begin{align*}
Q(x)= f_{k-1}(x;1,b^k)=\sum_{\alpha\in \oc_{\leq k-1}}(-1)^{\left\lfloor \frac{\ell(\alpha)+1}{2}\right\rfloor}b^{k\cdot\ell(\alpha)}x^{S_{k-1}(\alpha)}.
    \end{align*}
\end{theorem}    
  \begin{proof}
By Proposition \ref{prop:vk} and Equation \eqref{eq:determ}, the denominator of $V_k(x;1,b)$ is
   \begin{align*}
     \det(I-A(x))&=D_{k-1}(b^k x,b^k x^2,\ldots,b^k x^{k-1})\\
     &=\varphi_{k-1}(b^k x,b^k x^2,\ldots,b^k x^{k-1})\\
     &=f_{k-1}(x;1,b^k).\qedhere 
\end{align*}
\end{proof}

\section{A refinement of the combinatorial formula}\label{sec:refinement}
In this section we obtain an explicit expression for the polynomial $Q(x)$ of Theorem \ref{thm:vkden}. We adopt the convention that the binomial coefficient ${n \choose k}$ is 0 if either $n$ or $k$ is not an integer.
\begin{lemma}\label{lem:oddintok}
  The number of compositions of $n$ into $k$ odd parts equals
  \begin{align*}
    {\frac{n+k}{2}-1 \choose k-1}.
  \end{align*}
\end{lemma}
\begin{proof}
  The number of such compositions equals the number of solutions to
  \begin{align*}
    (2b_1-1)+(2b_2-1)+\cdots+(2b_k-1)=n,
  \end{align*}
  in positive integers $b_i$. Since the last equation is equivalent to $\sum_{i=1}^k b_i=(n+k)/2$, the result follows.
\end{proof}
By summing over $n$, we obtain the following corollary. 
\begin{corollary}\label{cor:oddintokleqn} 
  The number of odd compositions of integers $\leq n$ into $k$ parts equals
  \begin{align*}
    {\lfloor \frac{n+k}{2} \rfloor \choose k}.
  \end{align*}
\end{corollary}

Let $N$ be a positive integer. Given an integer $s$, we would like to count the number of compositions $\alpha\in \oc_{\leq N}$ for which $S_N(\alpha)=s$. It will be convenient to visualize a composition in $\oc_{\leq N}$ by a diagram. For instance, when $N=15$, the composition $\alpha=(1,3,5,3,1)\in \oc_{\leq N}$ is represented by the diagram shown in Figure \ref{fig:zigzag}. From left to right, the labeled nodes correspond to the values $\sigma_N(i)\; (1\leq i\leq N)$ and, for each $1\leq i\leq \ell(\alpha)$, the node $\sigma_N(\alpha_1+\cdots+\alpha_i)$ is enclosed in a box. In this case, the integers $15,2,11,6,9$ are boxed. We have $S_N(\alpha)=15+2+11+6+9=43.$
\begin{figure}[!ht]
  \centering
  \includegraphics[scale=.7]{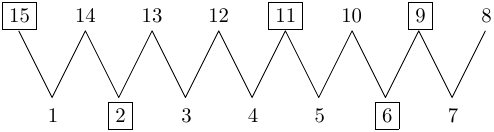} 
  \caption{The diagram of $\alpha=(1,3,5,3,1)\in \oc_{\leq 15}$.}
  \label{fig:zigzag}
\end{figure}

A \emph{slide} in the diagram of a composition is defined as selecting two successive boxes when scanning from left to right (for instance, boxes numbered 2 and 11 or those numbered 11 and 6) and moving them to the left by an equal number of positions. For example, sliding the boxes numbered 2 and 11 by one position to the left results in the composition $(1,1,5,5,1)$ whose diagram is shown in Figure~\ref{fig:slides}.
Observe that if $\tilde{\beta}$ is obtained from $\beta$ by sliding, then $S_N(\tilde{\beta})=S_N(\beta)$. By performing successive slides starting with $\alpha=(1,3,5,3,1)$ one eventually obtains the composition $(1,1,1,1,5)$ shown in Figure~\ref{fig:slides}.
\begin{figure}[!ht]
  \centering
  \includegraphics[scale=.7]{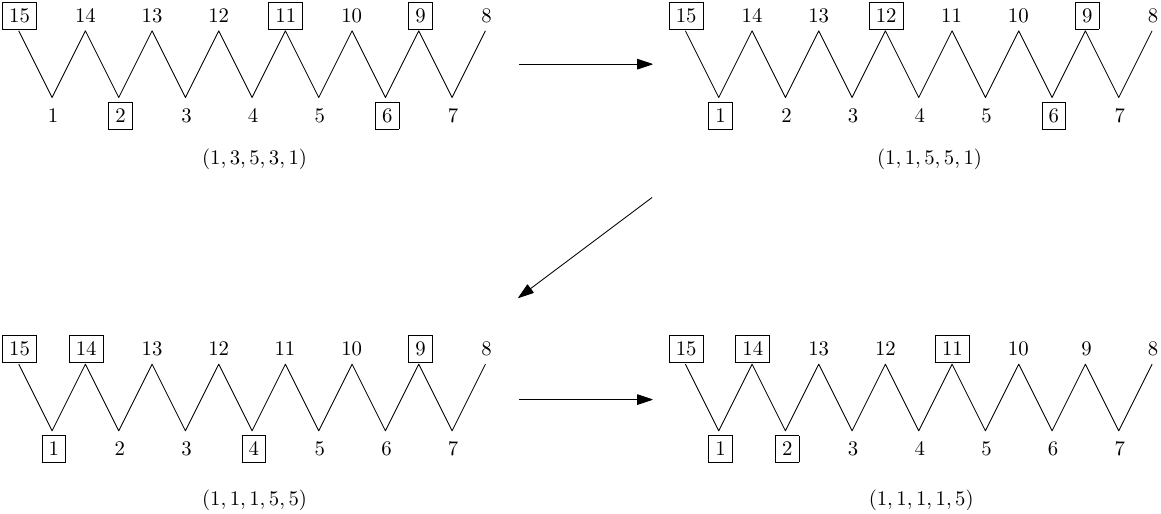} 
  \caption{Successive slides applied to $\alpha=(1,3,5,3,1)\in \oc_{\leq 15}$.}
  \label{fig:slides} 
\end{figure}
Suppose $\alpha=(\alpha_1,\ldots,\alpha_\ell)$. For each $i$ with $2\leq i\leq \ell-1$, sliding the $(i-1)$th and $i$th boxes as far to the left as possible corresponds to the mapping on $\oc_{\leq N}$ defined by
\begin{align*}
  (\alpha_1,\ldots,\alpha_{i-1},\alpha_i,\alpha_{i+1},\ldots,\alpha_\ell)\mapsto   (\alpha_1,\ldots,\alpha_{i-2},1,\alpha_i,\alpha_{i-1}+\alpha_{i+1}-1,\alpha_{i+2},\ldots,\alpha_\ell).
\end{align*}
On the other hand, sliding the $(\ell-1)$th and $\ell$th boxes as far to the left as possible corresponds to
\begin{align*}
  (\alpha_1,\ldots,\alpha_{\ell-2},\alpha_{\ell-1},\alpha_\ell)\mapsto   (\alpha_1,\ldots,\alpha_{\ell-2},1,\alpha_\ell).
\end{align*} 

Note that performing successive slides on any composition $\alpha\in \oc_{\leq N}$ eventually results in a composition of the form $(1,1,\ldots,1,a)$ for some odd integer $a$. We denote this composition by $\tilde{\alpha}$. One can express $\tilde{\alpha}$ explicitly in terms of $\alpha$. If $\alpha=(\alpha_1,\alpha_2,\ldots,\alpha_{2\ell+1})$, then
\begin{align*}
  \tilde{\alpha}=(1^{2\ell},\alpha_1+\alpha_3+\cdots+\alpha_{2\ell+1}-\ell).
\end{align*}
On the other hand, if $\alpha=(\alpha_1,\alpha_2,\ldots,\alpha_{2\ell})$, then
\begin{align*}
  \tilde{\alpha}=(1^{2\ell-1},\alpha_2+\alpha_4+\cdots+\alpha_{2\ell}-\ell+1).
\end{align*}
\begin{proposition}
If $\alpha,\beta\in {\rm OC}_{\leq N}$ are such that $\ell(\alpha)<\ell(\beta)$, then $S_N(\alpha)<S_N(\beta)$. 
\end{proposition}
\begin{proof}
It suffices to prove the proposition when $\ell(\beta)=\ell(\alpha)+1$. First suppose $\ell(\alpha)=2\ell$ is even. Then $\tilde{\alpha}=(1^{2\ell-1},a)$ and $\tilde{\beta}=(1^{2\ell},b)$ for some odd integers $a,b$. Writing $\sigma$ for $\sigma_N$, we have
  \begin{align*}
    S_N(\beta)-S_N(\alpha)&=    S_N(\tilde{\beta})-S_N(\tilde{\alpha})\\
                          &=\sigma(2\ell)+\sigma(2\ell+b)-\sigma(2\ell-1+a).
  \end{align*}
  Since $\sigma(i)>\sigma(j)$ whenever $i$ is odd and $j$ is even, it follows that $\sigma(2\ell+b)>\sigma(2\ell-1+a)$, implying that $S_N(\beta)>S_N(\alpha)$.

  Now suppose $\ell(\alpha)=2\ell+1$ is odd. Then $\tilde{\alpha}=(1^{2\ell},a)$ and $\tilde{\beta}=(1^{2\ell+1},b)$ for some odd integers $a,b$. In this case,
    \begin{align*}
    S_N(\beta)-S_N(\alpha)&=    S_N(\tilde{\beta})-S_N(\tilde{\alpha})\\
                          &=\sigma(2\ell+1)+\sigma(2\ell+1+b)-\sigma(2\ell+a),
    \end{align*}
which is positive since $\sigma(2\ell+1)\geq \sigma(2\ell+a)$.
\end{proof}
\begin{corollary}\label{cor:comps}
  If $S_N(\alpha)=S_N(\beta)$ for $\alpha,\beta \in {\rm OC}_{\leq N}$, then $\ell(\alpha)=\ell(\beta)$.  
\end{corollary}
 Setting $x_i=x^i$ for $1\leq i\leq N$ in the polynomial $\varphi_N(x_1,\ldots,x_N)$ of Definition \ref{def:varphir}, we obtain:
$$ \varphi_N(x,x^2,\ldots,x^N)= \sum_{\alpha\in \oc_{\leq N}}(-1)^{\left\lfloor \frac{\ell(\alpha)+1}{2}\right\rfloor}x^{S_N(\alpha)}.$$
The significance of Corollary \ref{cor:comps} is that no cancellation of terms occurs in the sum above. 

\begin{theorem}\label{thm:refinement}
  For each positive integer $N$, 
$$
\varphi_N(x,x^2,\ldots,x^N) = \sum_{\alpha\in \oc_{\leq N}}(-1)^{\left\lfloor \frac{\ell(\alpha)+1}{2}\right\rfloor}x^{S_N(\alpha)}=\sum_{s=0}^{\frac{N(N+1)}{2}} c_s(N)x^s,
$$
where
\begin{equation*}
  c_s(N)=\begin{cases}  
1 & \text{if }s=0,\\
(-1)^{q_1+1} \dbinom{N - r_1 }{q_1} \dbinom{q_1+r_1-\left\lfloor \frac{N}{2} \right\rfloor - 1 }{q_1} & \text{if } \lfloor\frac{2s}{N+1}  \rfloor \mbox{ is odd} \text{ and } r_1 > \lfloor\frac{N}{2}\rfloor, \\[8pt]
(-1)^{q_0} \dbinom{  r_0-1 }{ q_0-1} \dbinom{\left\lfloor \frac{N}{2} \right\rfloor - r_0 + q_0}{q_0} & \text{if } \lfloor\frac{2s}{N+1}  \rfloor \mbox{ is even} \text{ and } r_0 \leq \lfloor\frac{N}{2}\rfloor, \\
0 & \text{otherwise.}
\end{cases} 
\end{equation*}  
Here $q_j$ and $r_j$ are the quotient and remainder when $s$ is divided by $N+j$ for $j\in \{0,1\}$.
\end{theorem}

\begin{proof}
 By Corollary \ref{cor:comps}, the absolute value of $c_s(N)$ equals the cardinality of the set $\{\alpha \in \oc_{\leq N}:S_N(\alpha)=s\}.$ Since $S_N(\alpha)=0$ precisely when $\alpha$ is the empty composition, it follows that $c_0(N)=1$. For $s\geq 1$ we consider two cases. Suppose $s=S_N(\alpha)$ for some composition $\alpha$ of odd length, say $2\ell+1$. The diagram of $\tilde{\alpha}$ is shown in Figure \ref{fig:oddterm}.
\begin{figure}
  \centering
  \includegraphics[scale=.7]{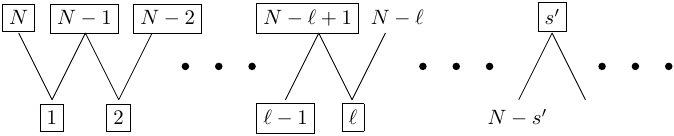}
  \caption{A composition $\tilde{\alpha}\in \oc_{\leq N}$ of odd length.}
  \label{fig:oddterm}
\end{figure}
We must have
\begin{align*}
  s=S_N(\tilde{\alpha})&=N+1+(N-1)+2+\cdots +(N-(\ell-1))+\ell+s'\notag \\
  &=(N+1)\ell+s',
\end{align*}
for some positive integer $s'$ satisfying $\lfloor N/2 \rfloor<s'\leq N-\ell$. Thus $\ell$ and $s'$ are respectively the quotient and remainder when $s$ is divided by $N+1$. Therefore, 
\begin{align}
  \ell(\alpha)=2\ell+1=\left\lfloor \frac{2s}{N+1} \right\rfloor. \label{eq:first} 
\end{align}
Note that $ \tilde{\alpha}=(1^{2\ell},a),$ where $a=2(N-\ell-s')+1$. A composition $(\alpha_1,\alpha_2,\ldots,\alpha_{2\ell+1})\in \oc_{\leq N}$ has canonical form $(1^{2\ell},a)$ precisely when the parts of $\alpha$ satisfy
\begin{align*}
  \alpha_1+\alpha_3+\cdots+\alpha_{2\ell+1}-\ell&=a,\\
  \alpha_2+\alpha_4+\cdots+\alpha_{2\ell}&\leq N-\ell-a.
\end{align*}
By Lemma \ref{lem:oddintok}, the number of solutions to the first equation above is
\begin{align*}
  {\frac{(\ell+a)+(\ell+1)}{2}-1 \choose \ell}={\ell+\frac{a-1}{2}\choose \ell}={N-s'\choose \ell}.
\end{align*}
By Corollary \ref{cor:oddintokleqn}, the number of solutions to the second equation is
\begin{align*}
  {\lfloor \frac{N-a}{2}  \rfloor\choose \ell}={\ell+s'-\lfloor \frac{N}{2} \rfloor-1\choose \ell}.
\end{align*}
As a result, the total number of compositions $\alpha\in \oc_{\leq N}$ with $S_N(\alpha)=s$ is given by
\begin{align*}
  {N-s'\choose \ell} {\ell+s'-\lfloor \frac{N}{2} \rfloor-1\choose \ell}=  {N-r_1\choose q_1} {q_1+r_1-\lfloor \frac{N}{2} \rfloor-1\choose q_1}.
\end{align*}
The sign of $c_s(N)$ equals $(-1)^{\ell+1}= (-1)^{q_1+1}$.
\begin{figure} 
  \centering
  \includegraphics[scale=.7]{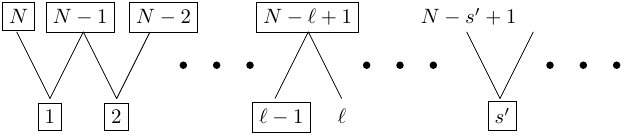}
  \caption{A composition $\tilde{\alpha}\in \oc_{\leq N}$ of even length.}
  \label{fig:eventerm}
\end{figure}

Now suppose $s=S_N(\alpha)$ for some composition $\alpha$ of even length, say $2\ell$. For such a composition $\alpha$, the diagram of $\tilde{\alpha}$ is shown in Figure \ref{fig:eventerm}. We have
\begin{align*}
  s&=S_N(\tilde{\alpha})=N+1+(N-1)+\cdots+(\ell-1)+(N-\ell+1)+s'\notag \\
  &=N\ell+s',
\end{align*}
where $s'$ satisfies $\ell\leq s'\leq \lfloor N/2 \rfloor$. Therefore $\ell$ and $s'$ are respectively the quotient and remainder when $s$ is divided by $N$. We also have
\begin{align}
  \ell(\alpha)=2\ell=\left\lfloor \frac{2s}{N+1} \right\rfloor. \label{eq:second} 
\end{align}
In this case $ \tilde{\alpha}=(1^{2\ell-1},a),$ where $a=2(s'-\ell)+1$. The number of compositions $\alpha=(\alpha_1,\alpha_2,\ldots,\alpha_{2\ell})\in \oc_{\leq N}$ which have canonical form $(1^{2\ell-1},a)$ equals the number of solutions to the system
\begin{align*}
  \alpha_2+\alpha_4+\cdots+\alpha_{2\ell}&=\ell-1+a=2s'-\ell,\\
  \alpha_1+\alpha_3+\cdots+\alpha_{2\ell-1}&\leq N+1-\ell-a=N+\ell-2s'.
\end{align*}
Again, by Lemma \ref{lem:oddintok} and Corollary \ref{cor:oddintokleqn} the number of solutions to the above system is given by
\begin{align*}
&  {s'-1 \choose \ell-1}{\ell-s'+\lfloor N/2 \rfloor\choose \ell}=  {r_0-1 \choose q_0-1}{q_0-r_0+\lfloor N/2 \rfloor\choose q_0}.
\end{align*}
Since the sign of $c_s(N)$ in this case is given by $(-1)^\ell=(-1)^{q_0}$, the theorem follows.
\end{proof}

\begin{corollary}\label{cor:length}
  For each composition $\alpha\in \oc_{\leq N}$, the length of $\alpha$ equals
  \begin{align*}
    \ell(\alpha)=\left\lfloor\frac{2S_N(\alpha)}{N+1}  \right\rfloor.
  \end{align*}
\end{corollary}
\begin{proof}
  This follows from Equations \eqref{eq:first} and \eqref{eq:second}.
\end{proof}
By Theorem \ref{thm:vkden}, we obtain the following result.
\begin{corollary}\label{cor:ffden} 
We have $V_k(x;1,b)=P(x)/f_{k-1}(x;1,b^k)$, where
    \begin{align*}
f_N(x;1,b) = \sum_{s=0}^{\frac{N(N+1)}{2}} c_s(N)\,  b^{\left\lfloor\frac{2s}{N+1}  \right\rfloor} x^s,
      \end{align*}
with $c_s(N)$ as in Theorem \ref{thm:refinement}.
\end{corollary}
The values of $f_N(x;1,1)=\varphi_N(x,x^2,\ldots,x^N)$ for $N\leq 6$ are shown in Table \ref{tab:seqfn}.
\begin{table}[h!]
\centering
\begin{tabular}{|c|l|}
\hline
$N$ & $f_N(x;1,1)$ \\ \hline
1 & $1 - x$ \\ 
2 & $1 - x^{2} - x^{3}$ \\ 
3 & $1 - x^{2} - x^{3} - x^{4} + x^{6}$ \\ 
4 & $1 - x^{3} - x^{4} - 2x^{5} - x^{6} + x^{8} + x^{10}$ \\ 
5 & $1 - x^{3} - x^{4} - x^{5} - 2x^{6} - x^{7} + 2x^{9} + 2x^{10} + x^{12} - x^{15}$ \\ 
6 & $1 - x^{4} - x^{5} - x^{6} - 3x^{7} - 2x^{8} - x^{9} + 2x^{11} + 2x^{12} + 3x^{14} + 2x^{15} - x^{18} - x^{21}$ \\ \hline
\end{tabular}
\caption{The first few polynomials $f_N(x;1,1).$}
\label{tab:seqfn}
\end{table}

Write $[N]$ for the set of the first $N$ positive integers.
\begin{definition}\label{def:palpha}
For a composition $\alpha=(\alpha_1,\ldots,\alpha_\ell)\in \oc_{\leq N}$, consider the subset $P_\alpha:=\{\sigma_N(\alpha_1+\cdots+\alpha_j)\}_{1\leq j\leq \ell}\subseteq [N]$.  
\end{definition}
  The following characterization of subsets of the form $P_\alpha$ follows from composition diagrams and will be used later.
\begin{proposition}\label{prop:palpha}
A nonempty subset $S=\{x_1>x_2>\cdots >x_\ell\}$ of positive integers is of the form $P_\alpha$ for some $\alpha \in \oc_{\leq N}$ precisely when $x_i+x_{\ell+1-i}\geq N+1$ and $x_{i+1}+x_{\ell+1-i}<N+1$ for $1\leq i\leq \ell$. Here we set $x_{\ell+1}=0$.
\end{proposition}
\begin{proof}
Let $\alpha\in\oc_{\le N}$ and $s_i=\alpha_1+\cdots+\alpha_i$. Since each $\alpha_i$ is odd, $s_i\equiv i\pmod 2$, so in $P_\alpha$ the values $\sigma_N(s_{2j-1})$ appear in decreasing order and $\sigma_N(s_{2j})$ in increasing order. If $P_\alpha=\{x_1>\cdots>x_\ell\}$, then $x_j=\sigma_N(s_{2j-1})$ and $x_{\ell+1-j}=\sigma_N(s_{2j})$. It is evident from the composition diagram that
\[\sigma_N(s_i)+\sigma_N(s_{i+1})\ge N+1\ (i\text{ odd}),\qquad
\sigma_N(s_i)+\sigma_N(s_{i+1})<N+1\ (i\text{ even})\]
which implies
\[x_j+x_{\ell+1-j}\ge N+1,\qquad x_{j+1}+x_{\ell+1-j}<N+1.\]

The converse follows similarly from the same composition diagram inequalities.
\end{proof}

\section{Special rim-hook tableaux and Hadamard products}\label{sec:hadamard-tilings}
In view of Proposition \ref{prop:Hkvkandhadamard}, our next goal is to derive an explicit formula for the Hadamard product  
\begin{align*}
  H(x):=\frac{1}{p_k(x;a,b)}*\frac{1}{p_k(x;a,b)}, 
\end{align*}
where $p_k(x;a,b)=1-ax-bx^k$. 
\begin{lemma}\label{lem:distinct}
Suppose $k\geq 2$ and $a,b$ are indeterminates over $\QQ$. Write $p_k(x;a,b)=1-ax-bx^k=\prod_{i=1}^k(1-\gamma_i x)$ where the $\gamma_i$ lie in an algebraic closure of $\QQ(a,b)$. Then the numbers $\gamma_i\;(1\leq i\leq k)$ are nonzero and pairwise distinct. Moreover, if $\gamma_i\gamma_j=\gamma_s\gamma_t$ for some integers $1\leq i,j,s,t\leq k$, then either $(i,j)=(s,t)$ or $(i,j)=(t,s)$.
\end{lemma}
\begin{proof}
Set $q(z)=z^k-az^{k-1}-b$. Then $q(\gamma_r)=0$ for $1\leq r\leq k$. Since $q(0)=-b\neq 0$, every $\gamma_r$ is nonzero. We first show that the roots $\gamma_1,\ldots,\gamma_k$ are pairwise distinct. If $\alpha$ were a repeated root of $q$, then $q(\alpha)=q'(\alpha)=0$. Now
\[
q'(z)=kz^{k-1}-(k-1)az^{k-2}=z^{k-2}(kz-(k-1)a).
\]
Because $\alpha\neq 0$, the equation $q'(\alpha)=0$ forces $\alpha=(k-1)a/k$. Substituting this into $q$ gives
\[
q\Big(\frac{k-1}{k}a\Big)=-\frac{(k-1)^{k-1}}{k^k}a^k-b\neq 0
\]
in $\QQ(a,b)$, since $a$ and $b$ are algebraically independent. Thus $q$ and $q'$ have no common root, so $q$ has no repeated root. Hence the $\gamma_r$ are pairwise distinct.

Now suppose $\gamma_i\gamma_j=\gamma_s\gamma_t$. Since $\gamma_i$ and $\gamma_t$ are nonzero, we may define
\[
\theta=\frac{\gamma_s}{\gamma_i}=\frac{\gamma_j}{\gamma_t}.
\]
Using $q(\gamma_i)=0$ and $q(\gamma_s)=q(\theta\gamma_i)=0$, we obtain
\begin{align*}
0&=q(\theta\gamma_i)-\theta^k q(\gamma_i)\\
&=\big((\theta\gamma_i)^k-a(\theta\gamma_i)^{k-1}-b\big)-\theta^k\big(\gamma_i^k-a\gamma_i^{k-1}-b\big)\\
&=a\theta^{k-1}(\theta-1)\gamma_i^{k-1}+b(\theta^k-1).
\end{align*}
Similarly, using $q(\gamma_t)=0$ and $q(\gamma_j)=q(\theta\gamma_t)=0$,
\[
a\theta^{k-1}(\theta-1)\gamma_t^{k-1}+b(\theta^k-1)=0.
\]
Subtracting these two relations yields
\[
a\theta^{k-1}(\theta-1)(\gamma_i^{k-1}-\gamma_t^{k-1})=0.
\]
If $\theta=1$, then $\gamma_s=\gamma_i$ and $\gamma_j=\gamma_t$. By the distinctness of the roots, this implies $s=i$ and $j=t$, so $(i,j)=(s,t)$.

Assume now that $\theta\neq 1$. Since $a\neq 0$ and $\theta\neq 0$, the displayed identity implies $\gamma_i^{k-1}=\gamma_t^{k-1}$. Set $\zeta=\gamma_t/\gamma_i$. Since $\gamma_i\neq 0$, this is well defined and satisfies $\zeta^{k-1}=1$. Because $q(\gamma_t)=q(\zeta\gamma_i)=0$, we have
\begin{align*}
0=q(\zeta\gamma_i)&=(\zeta\gamma_i)^k-a(\zeta\gamma_i)^{k-1}-b\\
&=\zeta^k\gamma_i^k-a\zeta^{k-1}\gamma_i^{k-1}-b\\
&=\zeta\gamma_i^k-a\gamma_i^{k-1}-b\\
&=(\zeta-1)\gamma_i^k.
\end{align*}
 Since $\gamma_i\neq 0$, it follows that $\zeta=1$, so $\gamma_t=\gamma_i$. Distinctness of the roots then gives $t=i$, and the identity $\gamma_i\gamma_j=\gamma_s\gamma_t$ becomes $\gamma_j=\gamma_s$, hence $j=s$. Therefore $(i,j)=(t,s)$, as required.
\end{proof}
\begin{corollary}
  The numbers $\gamma_i \gamma_j\;(1\leq i\leq j\leq k)$ are all distinct.
\end{corollary}

\begin{proposition}\label{prop:denom}
Let $p_k(x;a,b)=1-ax-bx^k=\prod_{i=1}^k(1-\gamma_i x)$, and suppose $H(x)=1/p_k(x;a,b)*1/p_k(x;a,b)=H_1(x)/H_2(x)$ where $H_1(x),H_2(x)$ are coprime polynomials in $\QQ(a,b)[x]$. Then (up to unit multiples) 
  \begin{align*}
H_2(x)= \prod_{i=1}^k(1-\gamma_i^2x)\prod_{1\leq i<j\leq k}(1-\gamma_i\gamma_jx).
  \end{align*}
\end{proposition}
\begin{proof}
Let $p_k'(x;a,b)$ denote the derivative of $p_k(x;a,b)$ with respect to $x$. Consider the partial fraction decomposition
\begin{align*}
  \frac{1}{p_k(x;a,b)}=\sum_{i=1}^k \frac{\beta_i}{1-\gamma_ix},
\end{align*}
where $\beta_i=-\gamma_i/p_k'(1/\gamma_i;a,b)\neq 0$ for $1\leq i\leq k.$  We have
\begin{align*}
  \frac{1}{p_k(x;a,b)}*\frac{1}{p_k(x;a,b)}&=\sum_{i=1}^k\sum_{j=1}^k \frac{\beta_i}{1-\gamma_ix}* \frac{\beta_j}{1-\gamma_jx}\\
                               &=\sum_{i=1}^k\sum_{j=1}^k \frac{\beta_i \beta_j}{1-\gamma_i\gamma_jx}\\
  &=\sum_{i=1}^k \frac{\beta_i^2}{1-\gamma_i^2 x}+\sum_{1\leq i<j\leq k}\frac{2\beta_i\beta_j}{1-\gamma_i\gamma_jx}.
\end{align*}
The $\beta_i$ are nonzero and the denominators in the partial fraction decomposition above are all distinct by Lemma \ref{lem:distinct}. It follows that the denominator of $H(x)$ when it is expressed as a quotient of coprime polynomials is given by the stated product up to a constant multiple. 
\end{proof}

\begin{corollary}\label{cor:Hnumdegree} 
If $p_k(x;a,b)=1-ax-bx^k$ and $H(x)=1/p_k(x;a,b)*1/p_k(x;a,b)=H_1(x)/H_2(x)$ with $H_1(x),H_2(x)$ coprime in $\QQ(a,b)[x]$, then  $\deg H_1(x)\leq {k \choose 2}$. 
\end{corollary}
\begin{proof}
By Proposition \ref{prop:Hkvkandhadamard} we have $H(x)=(1-J(x))^{-1}$ where $J(x)=P(x)/Q(x)$ is a quotient of (not necessarily coprime) polynomials. By Corollary \ref{cor:ffden}, we have  $\deg Q(x)={k \choose 2}$ which implies the result.
\end{proof}
Our next objective is to derive a formula for the denominator of $H(x)$ above by using ideas from the theory of symmetric functions. We begin with some preliminaries.
\subsection{Almost self-conjugate partitions}
A partition of a nonnegative integer $N$ is a sequence $\lambda=(\lambda_1,\lambda_2,\ldots)$ of nonnegative integers with sum $N$. One often omits trailing zeroes and simply writes $\lambda=(\lambda_1,\lambda_2,\ldots,\lambda_\ell)$ where $\lambda_i\;(1\leq i\leq \ell)$ are positive integers, called the parts of $\lambda$. We write $|\lambda|=N$ or $\lambda \vdash N$ to indicate that $\lambda$ is a partition of $N$. The Young diagram of a partition $\lambda$ is a visual representation of $\lambda$ as a collection of cells arranged in left-justified rows with $\lambda_i$ cells in the $i$th row, denoted ${\rm dg}(\lambda)$. The Young diagram of the partition $\lambda=(7,6,4,2,2,1)$ is shown in Fig. \ref{fig:young}. The conjugate of a partition, denoted $\lambda'$ is the partition whose Young diagram is obtained from that of $\lambda$ by interchanging rows and columns. The Durfee size or rank of $\lambda$ is the largest positive integer $i$ such that $\lambda_i\geq i$, denoted by ${\rm rank}(\lambda)$. If ${\rm rank}(\lambda)=r$, then the  Frobenius coordinates of $\lambda$ are given by the sequences $(\lambda_1-1,\lambda_2-2,\ldots,\lambda_r-r\mid \lambda'_1-1,\lambda'_2-2,\ldots,\lambda'_r-r)$ (see Figure \ref{fig:young}).
\begin{definition}(Dong and Wachs \cite[p. 4]{MR1912799})
An \emph{almost self-conjugate} (ASC) partition is a partition $\lambda$ satisfying $\lambda_i=\lambda'_i+1$ for $1\leq i\leq {\rm rank}(\lambda)$.  
\end{definition}

\begin{figure}[ht]
  \centering 
  \includegraphics[scale=.15]{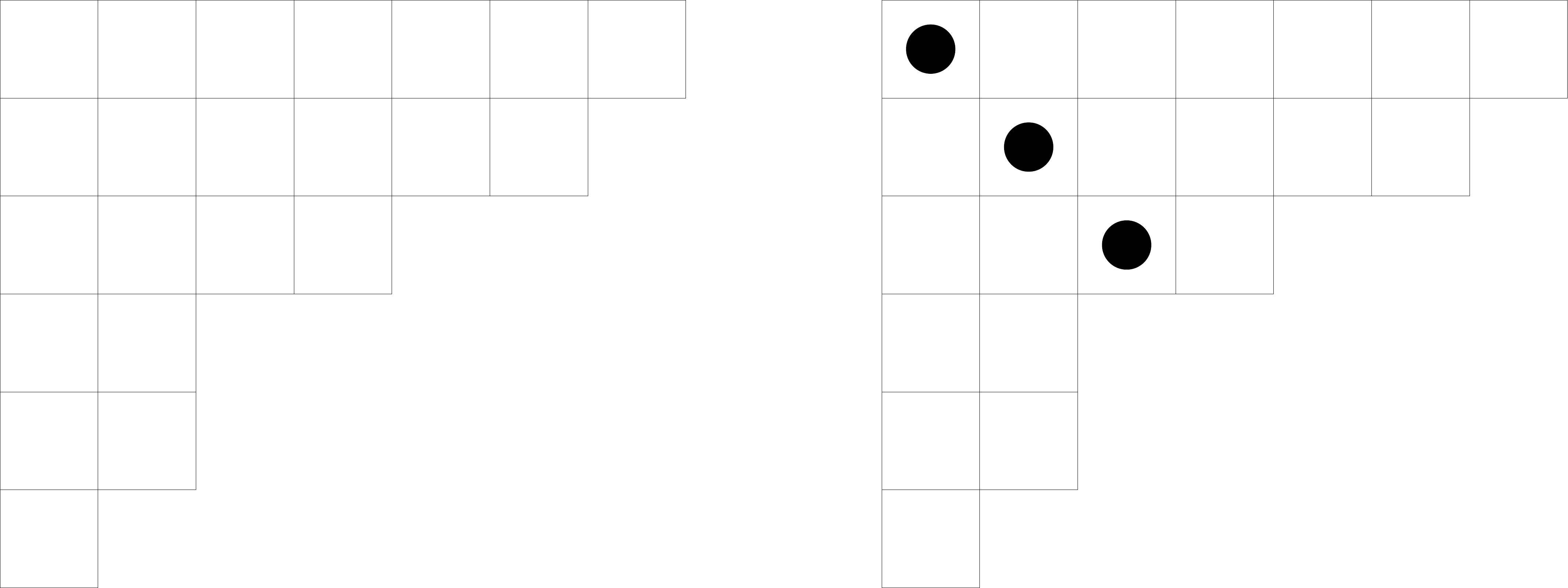}     
  \caption{Left: The Young diagram of $\lambda=(7,6,4,2,2,1)$, an almost self-conjugate (ASC) partition of rank 3. Right: The Frobenius coordinates of $\lambda$ are $(6,4,1\mid 5,3,0)$.}
  \label{fig:young}
\end{figure}
\begin{remark}
    Note that an ASC partition $\lambda$ of rank $r$ is uniquely determined by the first sequence $(\lambda_1-1,\ldots,\lambda_r-r)$ of its Frobenius coordinates, which is a tuple of distinct positive integers. Therefore, $|\lambda|=2\sum_{i=1}^r (\lambda_i-i)$. It follows that, for any nonnegative integer $N$, the map $\lambda \mapsto (\lambda_1-1,\ldots,\lambda_r-r)$ where $r={\rm rank}(\lambda)$ is a bijection from ASC partitions of $2N$ to partitions of $N$ into distinct parts.
  \end{remark}
\subsection{Symmetric functions and plethysms}
Let $\Lambda$ denote the algebra of formal symmetric functions in infinitely many variables $x=(x_1,x_2,\ldots)$. The algebra $\Lambda$ has several natural bases indexed by integer partitions, including the monomial symmetric functions $m_\lambda$, the power sum symmetric functions $p_\lambda$, the elementary symmetric functions $e_\lambda$, the complete homogeneous symmetric functions $h_\lambda$, and the Schur functions $s_\lambda$. 

For a formal power series $F$, denote by $p_k\circ F$ the series obtained by replacing each variable in $F$ with its $k$th power. For any symmetric function $g \in \Lambda$, the plethystic substitution $g\circ F$ is defined by first expressing $g$ as a polynomial in the power sum symmetric functions $p_r$ and then substituting $p_r\circ F$ for each occurrence of $p_r$ in $g$. Plethysms arise in the representation theory of symmetric and general linear groups in the context of composing representations. See Macdonald \cite[pg.\ 135]{MR1354144}  for more on plethysm.

Given the polynomial $p_k(x;a,b)=1-ax-bx^k=\prod_{i=1}^k(1-\gamma_ix)$, we wish to find a combinatorial formula for the product
  \begin{align*}
    \prod_{i=1}^k(1-\gamma_i^2x)\prod_{1\leq i<j\leq k}(1-\gamma_i\gamma_jx),
  \end{align*}
  which appears in the denominator (Proposition \ref{prop:denom}) of the Hadamard product $1/p_k(x;a,b)*1/p_k(x;a,b)$. Write $\gamma^{(2)}$ for the collection of pairwise products $\gamma_i\gamma_j \;(1\leq i<j\leq k)$. If $e_s(\gamma^{(2)})$ denotes the $s$th elementary symmetric function in the variables $\gamma^{(2)}$, then 
  \begin{align*}
    \prod_{1\leq i<j\leq k}(1-\gamma_i\gamma_jx)=\sum_{s=0}^{k \choose 2}(-1)^s e_s(\gamma^{(2)})x^s    =\sum_{s=0}^{k \choose 2}(-1)^s (e_s\circ e_2)(\gamma)x^s.
  \end{align*}
  The last equality follows by first expressing $e_s$ as a polynomial $g(p_1,\ldots,p_s)$ in the power sums $p_1,\ldots,p_s$ together with
  \[
  e_s(\gamma^{(2)})=g\bigl(p_1(\gamma^{(2)}),\ldots,p_s(\gamma^{(2)})\bigr)
  =g\bigl((p_1\circ e_2)(\gamma),\ldots,(p_s\circ e_2)(\gamma)\bigr)=(e_s\circ e_2)(\gamma).
  \]

  \begin{definition}
  A partition $\lambda$ is called \emph{threshold} if $\lambda'_i=\lambda_i+1$ for $1\leq i\leq {\rm rank}(\lambda)$.    
  \end{definition}

  It follows that a partition is threshold if and only if its conjugate is ASC. The following identity of Littlewood \cite[Eq.~(11.9;4)]{Littlewood1940} gives the Schur expansion of this plethystic composition of elementary symmetric functions; see also Klivans and Reiner \cite[Thm.~5.2]{MR2383434} and Colmenarejo, Orellana, Saliola, Schilling and Zabrocki \cite[Thm.~5.7]{MR4780733}.  
  \begin{theorem}\label{thm:threshold}
    We have
    \begin{align*}
      e_s\circ e_2=\sum_{\substack{\lambda \vdash 2s \\ \lambda \text{ threshold}}}s_{\lambda}.
    \end{align*}
  \end{theorem}
  The next step is to express the Schur functions in terms of elementary symmetric functions in order to obtain the elementary expansion of $e_s\circ e_2(\gamma)$. The coefficients in this expansion do not seem to be known in general, but we will show that a formula can be given in the current setting. The main idea is that the defining equation $p_k(x;a,b)=1-ax-bx^k=\prod_{i=1}^k(1-\gamma_ix)$ implies that several of the elementary symmetric functions in the $\gamma_i$ vanish and this fact can be exploited together with the fact that the Schur functions in Theorem \ref{thm:threshold} are indexed by threshold partitions. 

  The inverse Kostka numbers $K'_{\lambda\mu}$ are defined as the coefficients in the elementary expansion of the Schur function:
  \begin{equation}\label{eq:invk}
    s_{\lambda'}=\sum_{\mu}K'_{\mu\lambda}e_\mu.
  \end{equation}
We now describe a combinatorial formula for the inverse Kostka numbers. Given integer partitions $\mu,\nu$ such that $\nu_i\leq \mu_i$ for each $i\geq 1$, the skew shape $\mu/\nu$ is the collection of cells obtained from the Young diagram of $\mu$ by removing the squares corresponding to the Young diagram of $\nu$. A \emph{ribbon} is a skew shape that can be obtained by starting with a single cell and adding new cells such that each new cell is immediately to the left or directly below the previously added cell and shares an edge with it. We also allow for the empty ribbon with 0 cells. Equivalently, ribbons are connected skew shapes that do not contain four cells arranged in a $2\times 2$ square. A ribbon containing $k$ cells is called a $k$-ribbon. The sign of a ribbon spanning $r$ rows is defined as $(-1)^{r-1}$ (see Figure \ref{fig:ribbon}). A rim-hook is a ribbon whose cells can be deleted from the Young diagram of a partition to obtain another Young diagram.

  \begin{figure}[ht]
  \centering
  \includegraphics[scale=.15]{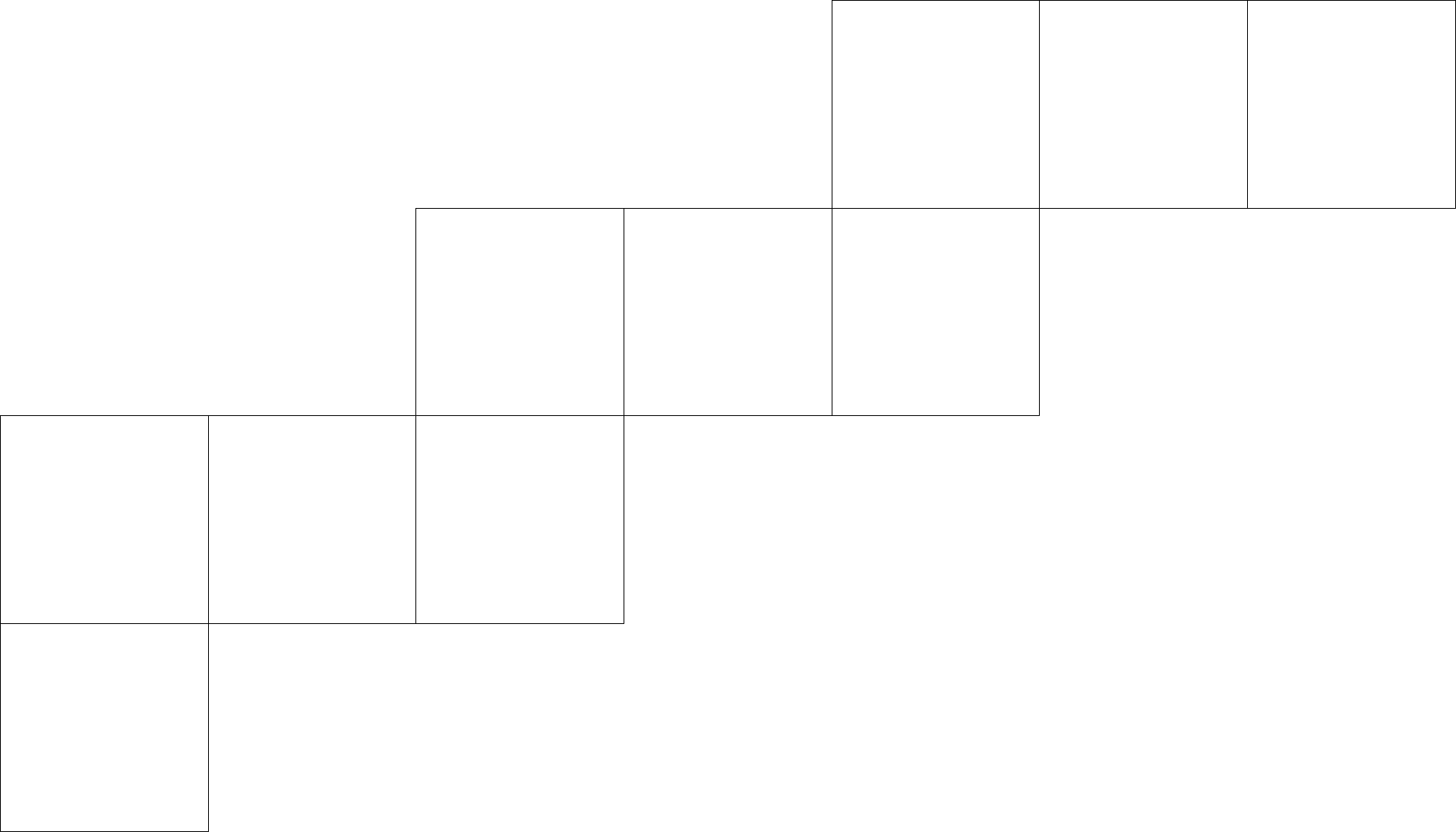}     
  \caption{The skew shape $(7,5,3,1)/(4,2)$ is a 10-ribbon spanning 4 rows and 7 columns. The ribbon has sign $-1$.}  
  \label{fig:ribbon}
\end{figure} 

  \begin{figure}[ht]
  \centering
  \includegraphics[scale=1]{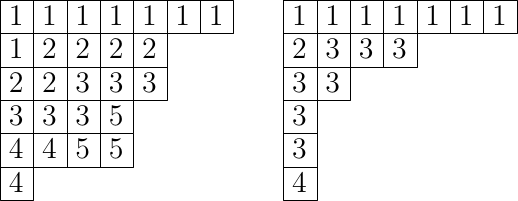}      
  \caption{Left: A rim-hook tableau of shape $(7,5,5,4,4,1)$, content $(8,6,6,3,3)$ and sign $-1$. The ribbon with cells numbered 5 is a rim-hook. Right: A special rim-hook tableau with sign $-1$.}
  \label{fig:rht} 
\end{figure}  
\subsection{Special rim-hook tableaux}\label{subsec:srht}
  \begin{definition} 
  Given partitions $\lambda$ and $\mu$ with $\mu\subseteq \lambda$, and a sequence $\alpha=(\alpha_1,\ldots,\alpha_\ell)$ of nonnegative integers, a \emph{rim-hook tableau} of shape $\lambda/\mu$ and content $\alpha$ is a sequence of partitions 
  $$
  \mu = \nu^0 \subseteq \nu^1 \subseteq \nu^2 \subseteq \cdots \subseteq \nu^\ell = \lambda,
  $$
such that each skew shape $\nu^i / \nu^{i-1}$ forms an $\alpha_i$-ribbon for $1 \leq i \leq \ell$. This tableau is represented visually by drawing the diagram of $\lambda/\mu$ and labeling each cell of the ribbon $\nu^i / \nu^{i-1}$ with the number $i$ as shown in Figure \ref{fig:rht}. 
  \end{definition}
The \textit{sign} of a rim-hook tableau $T$, denoted ${\rm sgn}(T)$ is given by the product of the signs of its ribbons $\nu^i / \nu^{i-1}$. Given a sequence $\alpha$ of nonnegative integers, let ${\rm sort}(\alpha)$ denote the partition obtained by sorting $\alpha$ in weakly decreasing order. 
\begin{definition}
  For partitions $\lambda$ and $\mu$ of $N$, a \emph{special rim-hook tableau} of shape $\lambda$ and \emph{type} $\mu$ is a rim-hook tableau $S$ of shape $\lambda$ and content $\alpha$ satisfying ${\rm sort}(\alpha)=\mu$ with the additional property that each nonzero ribbon in $S$ contains a cell in the first column of the diagram of $\lambda$. Let ${\rm SRHT}(\lambda,\mu)$ denote the collection of all special rim-hook tableaux of shape $\lambda$ and type $\mu$. 
\end{definition}
The following result gives a combinatorial formula for the inverse Kostka numbers (Loehr \cite[Thm.\ 11.64]{MR2777360}).
\begin{theorem}\label{thm:invk}
  If $\lambda,\mu$ are partitions of $N,$ then
  \begin{align*}
    K'_{\mu\lambda}=\sum_{S\in {\rm SRHT}(\lambda,\mu)}{\rm sgn}(S).
  \end{align*}
\end{theorem}

Combining Theorem \ref{thm:invk} with Theorem \ref{thm:threshold} yields the following result.
\begin{theorem} \label{thm:amu}
  For each nonnegative integer $s$,
  \begin{align*}
    e_s\circ e_2=\sum_{\mu\vdash 2s} a_\mu e_\mu,
  \end{align*}
  where
  \begin{align*}
    a_\mu=\sum_{\substack{\lambda \vdash 2s\\\lambda \text{ is ASC}}}\sum_{S\in {\rm SRHT}(\lambda,\mu)}{\rm sgn}(S).
  \end{align*}
\end{theorem}
\begin{proof}
    By Theorem \ref{thm:threshold} and Equation \eqref{eq:invk},
    \begin{align*}
        e_s \circ e_2 
        &= \sum_{\substack{\lambda \vdash 2s \\ \lambda \text{ threshold}}} \sum_{\mu \vdash 2s} K'_{\mu \lambda'} e_{\mu}  \\
        &= \sum_{\mu \vdash 2s} \sum_{\substack{\lambda \vdash 2s\\ \lambda \text{ is ASC}}}  K'_{\mu \lambda} e_{\mu},
    \end{align*}
since   \(\lambda\) is ASC iff \(\lambda'\) is threshold. 
The result now follows from Theorem \ref{thm:invk}.
\end{proof}
We will derive an explicit formula for the coefficients $a_\mu$ in the case where $\mu$ has only parts equal to 1 or $k$. Write $\mu=(k^q1^r)$ if $\mu$ has $q$ parts equal to $k$ and $r$ parts equal to 1. It can be seen that the number of cells in a ribbon equals $n_r+n_c-1$, where $n_r$ and $n_c$ equal the number of rows and columns spanned by the ribbon (see example in Figure~\ref{fig:ribbon}). Given a partition $\lambda$, index the rows and columns of ${\rm dg}(\lambda)$ from top to bottom and left to right respectively by positive integers. The cell appearing in row $i$ and column $j$ of ${\rm dg}(\lambda)$ is assigned coordinates $(i,j)$. The hook of the cell $(i,j)$ consists of the cell $(i,j)$ along with all cells to its right in the same row and all cells below it in the same column. The number of cells in the hook is called the hook length and is given by $h_{(i,j)}=(\lambda_i-j)+(\lambda'_j-i)+1$ (see Figure \ref{fig:hook}).  
    \begin{figure}[ht]
  \centering
  \includegraphics[scale=.6]{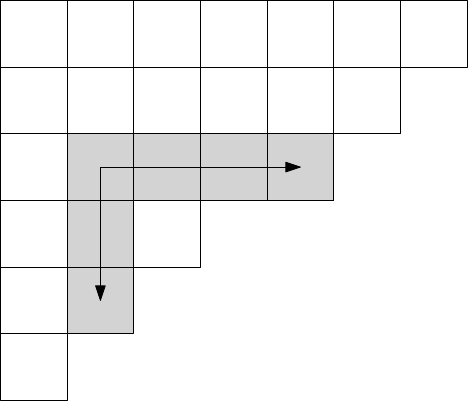}       
  \caption{The hook of the cell $(3,2)$ in the partition $(7,6,5,3,2,1)$. In this case $h_{(3,2)}=6$.}
  \label{fig:hook}   
\end{figure}

\begin{lemma}\label{lem:uniqmu}
Suppose $\lambda$ is an ASC partition and let $\mu=(k^q1^r)$ for some integer $k\geq 2$. If ${\rm SRHT}(\lambda,\mu)$ is nonempty, then ${\rm rank}(\lambda)=q$. Therefore, if ${\rm SRHT}(\lambda,\tilde{\mu})$ is nonempty for some partition $\tilde{\mu}$  all of whose parts are either 1 or $k$, then $\tilde{\mu}$ is uniquely determined by $\lambda$.
\end{lemma}
\begin{proof}
  Suppose ${\rm rank}(\lambda)=d$ and consider a tableau $S\in {\rm SRHT}(\lambda,\mu)$. Note that $\lambda_d\geq d+1$ and consider the $d$ cells $(i,i+1)\; (1\leq i\leq d)$ immediately to the right of the diagonal in the Young diagram of $\lambda$ (see Figure \ref{fig:q-ribbons}). Since the 1-ribbons in $S$ can only appear in the first column, it can be seen that each of these $d$ cells is contained in a different $k$-ribbon. Therefore $d\leq q$.
    \begin{figure}[ht]
  \centering
  \includegraphics[scale=.15]{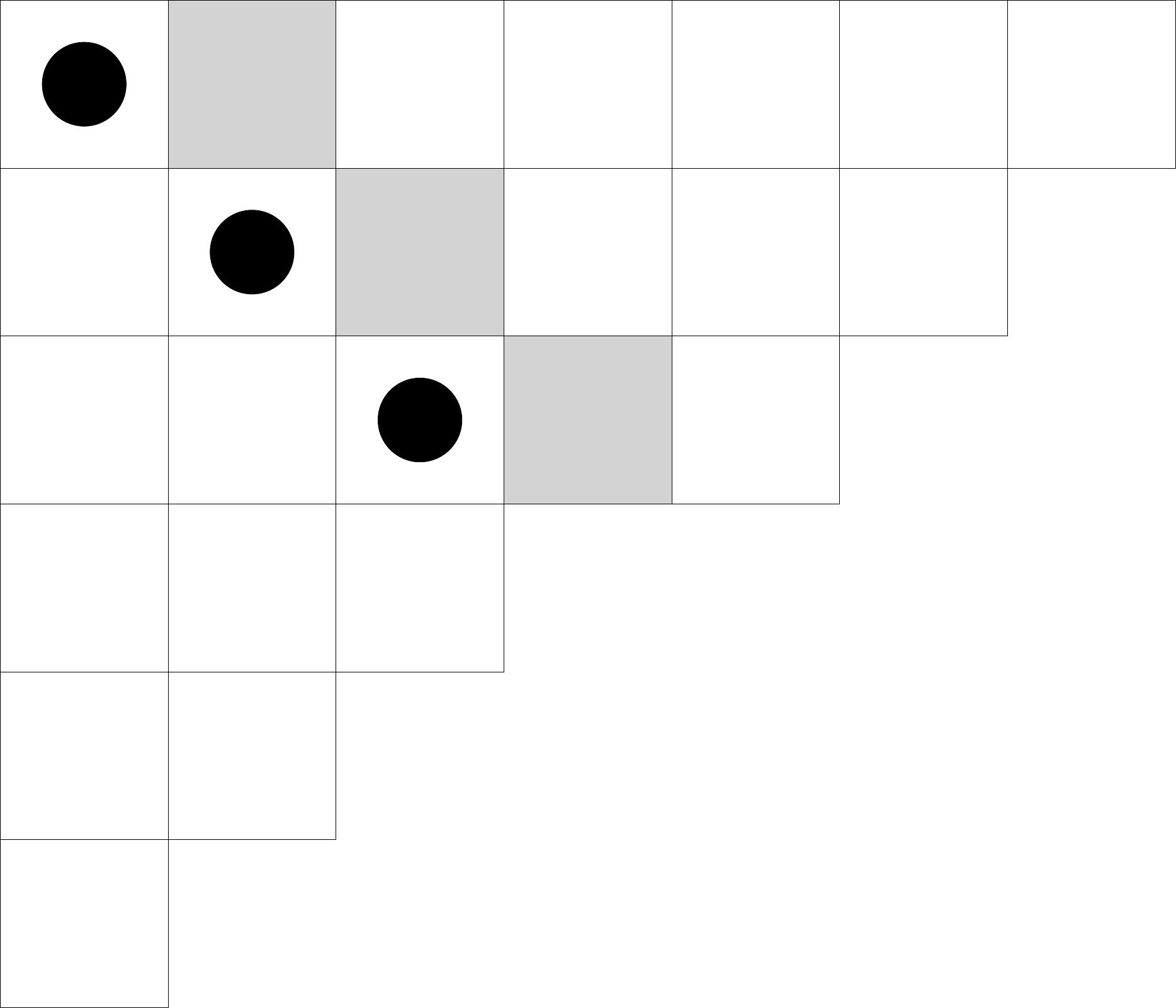}       
  \caption{Each of the shaded cells is contained in a different $k$-ribbon.}
  \label{fig:q-ribbons}  
\end{figure}

If possible, suppose $q>d$. We claim that the $k$-ribbon $R$ containing the cell $(1,2)$ also contains the cell $(1,\lambda_1)$. If not, then the cell $(1,\lambda_1)$ is contained in another $k$-ribbon which necessarily spans more rows and columns than $R$, contradicting the fact that their lengths are equal. Similar reasoning shows that the cells $(i,i+1)$ and $(i,\lambda_i)$ lie on the same $k$-ribbon for $1\leq i\leq d$. Observe that $\lambda_1\leq k$ since the $k$-ribbon containing $(1,\lambda_1)$ also contains a cell in the first column. Since $q>d,$ there exists a $k$-ribbon in $S$ which does not contain any cell in the first $d$ rows of ${\rm dg}(\lambda)$. The length of this $k$-ribbon is then at most $(\lambda'_1-d)+d-1=\lambda'_1-1=\lambda_1-2<k$, a contradiction. It follows that $q=d$.
\end{proof}
The following result gives a necessary and sufficient condition for ${\rm SRHT}(\lambda,\mu)$ to be nonempty.
\begin{theorem}\label{thm:frobconditions}
  Suppose $\lambda$ is an ASC partition of rank $d$ and let $\mu=(k^d1^e)$ be a partition of the same size with $k\geq 2$. If $\lambda$ has Frobenius coordinates $(f_1,\ldots,f_d\mid f_1-1,\ldots,f_d-1)$, then ${\rm SRHT}(\lambda,\mu)$ is nonempty if and only if the following conditions are satisfied.
  \begin{enumerate}
  \item   $f_j+f_{d+1-j}\geq k$ for $1\leq j\leq d$.
  \item $f_{j+1}+f_{d+1-j}< k$ for $1\leq j\leq d$.
  \end{enumerate}
  Here we set $f_{d+1}=0$, so the second condition above becomes $\lambda_1\leq k$ when $j=d$.
\end{theorem}
\begin{proof}
  Observe that ${\rm SRHT}(\lambda,\mu)$ is nonempty if and only if it is possible to start with ${\rm dg}(\lambda)$ and remove a sequence of rim-hooks, each of which is either a 1-ribbon or a $k$-ribbon containing a cell in the first column, and eventually obtain the empty partition. We have already seen that the $k$-ribbons removed in this process have their (unique) rightmost cells at the end of the first $d$ rows of ${\rm dg}(\lambda)$. Let $c_i$ denote the last cell in row $i$, for $1\leq i\leq d.$ Consider the longest rim-hook $R_d$ starting at the cell $c_d$; this rim-hook necessarily ends in the bottom cell of the first column of ${\rm dg}(\lambda)$ and has at least $k$ cells. Since this rim-hook has length $h_{(d,1)}$, it follows that $h_{(d,1)}\geq k$ which is equivalent to $f_1+f_d\geq k$. Further, if $h_{(d,1)}>k$, then it should be possible to remove $h_{(d,1)}-k$ 1-ribbons from the first column. This is possible if and only if $\lambda'_1-\lambda'_2\geq h_{(d,1)}-k=\lambda_d+\lambda'_1-d-k$. The last condition is equivalent to $f_2+f_d<k$.

  Note that removing $h_{(d,1)}-k$ 1-ribbons from the first column followed by removing the $k$-ribbon containing $c_d$ from ${\rm dg}(\lambda)$ is equivalent to removing the rim-hook $R_d$ from ${\rm dg}(\lambda)$. Therefore, the resulting partition $\lambda^{(2)}$ has initial column lengths $\lambda'_2-1,\lambda'_3-1,\ldots$ (in general, if the partition has column lengths $\ell_1,\ell_2,\ell_3,\ldots$, then removing the longest rim-hook in this manner results in the initial column lengths being $\ell_2-1,\ell_3-1,\ldots$). Consider the longest rim-hook $R_{d-1}$ starting at the cell $c_{d-1}$ in ${\rm dg}(\lambda^{(2)})$. Then this rim-hook has length $\lambda'_2-1-(d-1)+\lambda_{d-1}=f_2+f_{d-1}$ which is necessarily at least $k$. By the same argument above, it should be possible to remove $f_2+f_{d-1}-k$ 1-ribbons from the first column of ${\rm dg}(\lambda^{(2)})$. Therefore $\lambda'_2-\lambda'_3\geq f_2+f_{d-1}-k$, which is equivalent to $f_3+f_{d-1}<k$.

  Removing $f_2+f_{d-1}-k$ 1-ribbons from the first column of ${\rm dg}(\lambda^{(2)})$ followed by removing the $k$-ribbon containing $c_{d-1}$ is equivalent to removing the rim-hook $R_{d-1}$ from ${\rm dg}(\lambda^{(2)})$. The resulting partition $\lambda^{(3)}$ has initial column lengths $\lambda'_3-2,\lambda'_4-2,\ldots$.

  In general, after the longest rim-hooks $R_d,R_{d-1},\ldots,R_{d-j+1}$ (with $R_i$ beginning at $c_i$) have been removed for some $j<d$, the resulting partition $\lambda^{(j+1)}$ has initial column lengths $\lambda'_{j+1}-j,\lambda'_{j+2}-j,\ldots$. Therefore, the longest rim-hook starting at the cell $c_{d-j}$ has length $\lambda'_{j+1}-j-(d-j)+\lambda_{d-j}=f_{j+1}+f_{d-j}$ which is necessarily at least $k$. On the other hand, it should be possible to remove $f_{j+1}+f_{d-j}-k$ 1-ribbons from the first column of $\lambda^{(j+1)}$ which yields the condition $\lambda'_{j+1}-\lambda'_{j+2}\geq f_{j+1}+f_{d-j}-k$ which is equivalent to $f_{j+2}+f_{d-j}<k$. From the preceding discussion it is clear that the inequalities derived are necessary as well as sufficient for ${\rm SRHT}(\lambda,\mu)$ to be nonempty and the result follows.
\end{proof}

\begin{corollary}\label{cor:card1}
 Suppose $\lambda$ is an ASC partition of rank $d$ and let $\mu=(k^d1^e)$ be a partition of the same size with $k\geq 2$ such that ${\rm SRHT}(\lambda,\mu)$ is nonempty. Then $|{\rm SRHT}(\lambda,\mu)|=1$.
  \end{corollary}
  \begin{proof}
Consider the successive removal of ribbons of length $1$ or $k$ from ${\rm dg}(\lambda)$. The argument in the proof of Theorem \ref{thm:frobconditions} shows that, at each stage, if the relevant longest rim-hook $R_i$ has length $L>k$, then one must first remove $L-k$ $1$-ribbons from the bottom of the first column before a $k$-ribbon through $c_i$ can be removed. If $L=k$ and both removals preserve partition shape, only the $k$-ribbon removal can extend to a full special rim-hook tableau of type $\mu$; removing a $1$-ribbon first would leave no later $k$-ribbon through $c_i$. This constraint forces a deterministic decomposition sequence where the shape of $\lambda$ strictly dictates the length of the ribbon removed at each stage. Therefore the resulting tableau is unique. 
  \end{proof} 

\begin{definition}
  Let ${\rm ASC}_{1,k}(N)$ denote the set of all ASC partitions $\lambda$ of $N$ such that there exists a special rim-hook tableau of shape $\lambda$ and type $\mu=(k^q1^r)$ for some integers $q,r.$
\end{definition}
Recall Definition \ref{def:palpha} of $P_\alpha$. Write $r(\lambda)$ for the rank of $\lambda$.
\begin{proposition}\label{prop:bijection}
  For each nonnegative integer $s$ there is a bijection between ${\rm ASC}_{1,k}(2s)$ and the set of compositions $\alpha\in \oc_{\leq k-1}$ satisfying $S_{k-1}(\alpha)=s$.
\end{proposition}
\begin{proof}
  By Proposition \ref{prop:palpha} and Theorem \ref{thm:frobconditions} a partition $\lambda$ is an element of ${\rm ASC}_{1,k}(2s)$ if and only if the set  $\{\lambda_1-1,\lambda_2-2,\ldots,\lambda_{r(\lambda)}-r(\lambda)\}$ is of the form $P_\alpha$ for some composition $\alpha\in \oc_{\leq k-1}$.  Since $\sum_{i=1}^{r(\lambda)}(\lambda_i-i)=s$ and $\alpha$ is uniquely determined by $P_\alpha$, the result follows.
\end{proof}
\begin{remark}
If $\lambda\in {\rm ASC}_{1,k}(2s) $ maps to a composition $\alpha\in \oc_{\leq k-1}$ under the bijection of Proposition \ref{prop:bijection}, then  $\ell(\alpha)=r(\lambda)$.  
\end{remark}

\begin{theorem}\label{prop:denominator}
If $p_k(x;a,b)=1-ax-bx^k$ and $H(x)=1/p_k(x;a,b)*1/p_k(x;a,b)=H_1(x)/H_2(x)$ where $H_1(x)$ and $H_2(x)$ are coprime in $\QQ(a,b)[x]$, then
\[
H_2(x)=p_k(\sqrt{x};a,b)\; p_k(-\sqrt{x};a,b)\; f_{k-1}(x;a,-b),
\]
with $f_{k-1}(x;a,-b)$ as in Equation~\eqref{eq:fnweighted}.
\end{theorem}
\begin{proof}
By Proposition \ref{prop:denom}, the denominator of $H(x)$ equals
  \begin{align*}
    \prod_{i=1}^k(1-\gamma_i^2x)\prod_{1\leq i<j\leq k}(1-\gamma_i\gamma_jx).
  \end{align*}
  where the $\gamma_i$ are defined by $p_k(x;a,b)=\prod_{i=1}^k(1-\gamma_ix)$. First note that
  \begin{align*}
p_k(\sqrt{x};a,b)\; p_k(-\sqrt{x};a,b)=        \prod_{i=1}^k(1-\gamma_i^2x).
  \end{align*}
  On the other hand, we have
  \begin{align*}
    \prod_{1\leq i<j\leq k}(1-\gamma_i\gamma_jx)=\sum_{s=0}^{k \choose 2}(-1)^s \; e_s\circ e_2(\gamma)\; x^s.    
  \end{align*}
  By Theorem \ref{thm:amu}, 
    \begin{align*}
    e_s\circ e_2=\sum_{\mu\vdash 2s} \sum_{\substack{\lambda \vdash 2s\\\lambda \text{ is ASC}}}\sum_{S\in {\rm SRHT}(\lambda,\mu)}{\rm sgn}(S)\; e_\mu.
  \end{align*}
We have $e_1(\gamma)=a$, $ e_k(\gamma)=(-1)^{(k+1)}b,$ and $e_j(\gamma)=0$ for $1<j<k$. Therefore $e_\mu(\gamma)=0$ unless $\mu$ has all parts in $\{1,k\}$. By Lemma \ref{lem:uniqmu}, for each $\lambda\in {\rm ASC}_{1,k}(2s)$ there is a unique partition $\mu$ with parts in $\{1,k\}$ such that ${\rm SRHT}(\lambda,\mu)$ is nonempty, and by Corollary \ref{cor:card1} this set contains a unique tableau. Hence
    \begin{align*}
    e_s\circ e_2(\gamma)= \sum_{\lambda \in {\rm ASC}_{1,k}(2s)}{\rm sgn}(S_\lambda)\; e_\mu(\gamma),
  \end{align*}
where $\mu=(k^{r(\lambda)}1^{2s-k\cdot r(\lambda)})$ and $S_\lambda$ is the unique element of ${\rm SRHT}(\lambda,\mu)$. To compute ${\rm sgn}(S_\lambda)$, note that the 1-ribbons do not contribute to the sign while the $k$-ribbons contribute $\prod_{i=1}^{r(\lambda)}(-1)^{k-\lambda_i}=(-1)^{k\cdot r(\lambda)-D(\lambda)}$, where $D(\lambda)=\sum_{i=1}^{r(\lambda)}\lambda_i$. Therefore  
    \begin{align*}
      e_s\circ e_2(\gamma)&= \sum_{\lambda \in {\rm ASC}_{1,k}(2s)} (-1)^{k\cdot r(\lambda)+D(\lambda)} ((-1)^{k+1}b)^{r(\lambda)}a^{2s-k\cdot r(\lambda)}\\
      &= \sum_{\lambda \in {\rm ASC}_{1,k}(2s)} (-1)^{r(\lambda)+D(\lambda)} b^{r(\lambda)}a^{2s-k\cdot r(\lambda)}.
  \end{align*}
  In view of the bijection of Proposition \ref{prop:bijection}, 
  \begin{align*}
          e_s\circ e_2(\gamma)&=\sum_{\substack{\alpha\in \oc_{\leq k-1}\\ S_{k-1}(\alpha)=s}} (-1)^{\ell(\alpha)+s+{\ell(\alpha)+1 \choose 2}}\; b^{\ell(\alpha)}a^{2s-k\cdot \ell(\alpha)}.
  \end{align*}
It follows that
  \begin{align*}
    \sum_{s=0}^{k \choose 2} (-1)^s e_s\circ e_2(\gamma)x^s &=\sum_{\alpha\in \oc_{\leq k-1}} (-1)^{\lfloor \frac{\ell(\alpha)}{2} \rfloor}\; b^{\ell(\alpha)}a^{2S_{k-1}(\alpha)-k\cdot \ell(\alpha)}x^{S_{k-1}(\alpha)}\\
    &=\sum_{\alpha\in \oc_{\leq k-1}}(-1)^{\left\lfloor \frac{\ell(\alpha)}{2} \right\rfloor}b^{\ell(\alpha)}a^{\rm rem(\alpha)}x^{S_{k-1}(\alpha)},
  \end{align*}
by Corollary \ref{cor:length}. On the other hand, by \eqref{eq:fnweighted} we have
\[
f_{k-1}(x;a,-b)=\sum_{\alpha\in \oc_{\leq k-1}}(-1)^{\left\lfloor \frac{\ell(\alpha)+1}{2}\right\rfloor+\ell(\alpha)}a^{{\rm rem}(\alpha)}b^{\ell(\alpha)}x^{S_{k-1}(\alpha)}.
\]
Since
\[
\left\lfloor \frac{m+1}{2}\right\rfloor+m \equiv \left\lfloor \frac{m}{2}\right\rfloor \pmod 2
\]
for every integer $m\geq 0$, the last displayed sum equals
\[
\sum_{\alpha\in \oc_{\leq k-1}}(-1)^{\left\lfloor \frac{\ell(\alpha)}{2}\right\rfloor}a^{{\rm rem}(\alpha)}b^{\ell(\alpha)}x^{S_{k-1}(\alpha)}.
\]
Therefore $\prod_{1\leq i<j\leq k}(1-\gamma_i\gamma_jx)=f_{k-1}(x;a,-b),$ 
and the result follows.
\end{proof}

\begin{remark}
The reappearance of the polynomial $f_N(x;a,-b)$ in Theorem \ref{prop:denominator} is quite unexpected. It would be nice to have a conceptual explanation for this coincidence.
\end{remark}

  \section{Final Hadamard-product calculation and proof of the main theorem}\label{sec:final-hadamard-proof}
  In this section we prove Theorem \ref{thm:main} using a result of Graham on fault-free tilings. 
  \begin{lemma}\label{lem:finalAden}
Let $N\geq 2$, let $p_N(x;a,b)=1-ax-bx^N$, and let $J(x)$ be the unique formal power series satisfying
\[
\frac{1}{1-J(x)}=\frac{1}{p_N(x;a,b)}*\frac{1}{p_N(x;a,b)}.
\]
When $J(x)$ is written as a quotient of coprime polynomials in $\QQ(a,b)[x]$, its denominator is $f_{N-1}(x;a,b)$.
  \end{lemma}
  \begin{proof}
Let $c$ be an indeterminate, and consider the power series $V_N(x;a,c)$ of Equation~\eqref{eq:hkvk}. By Proposition \ref{prop:Hkvkandhadamard},
\[
\frac{1}{1-V_N(x;a,c)}=\frac{1}{p_N(x;a,c^N)}*\frac{1}{p_N(x;a,c^N)}.
\]
Thus the series $J(x)$ is obtained from $V_N(x;a,c)$ by replacing the indeterminate $c^N$ with $b$. By Remark \ref{rem:vkhomogeneity}, applied with $k=N$, we have $ V_N(x;a,c)=V_N(a^2x;1,ca^{-1}).$ Applying Theorem \ref{thm:vkden} with $k=N$ and $b=ca^{-1}$, we find that the denominator of $V_N(x;a,c)$ is $f_{N-1}(a^2x;1,(ca^{-1})^N).$ Using \eqref{eq:fnweighted}, we obtain
\[
f_{N-1}(a^2x;1,(ca^{-1})^N)=\sum_{\alpha\in \oc_{\leq N-1}}(-1)^{\left\lfloor \frac{\ell(\alpha)+1}{2}\right\rfloor}c^{N\ell(\alpha)}a^{2S_{N-1}(\alpha)-N\ell(\alpha)}x^{S_{N-1}(\alpha)}.
\]
Now Corollary \ref{cor:length}, applied with $N-1$ in place of $N$, gives
\[
\ell(\alpha)=\left\lfloor \frac{2S_{N-1}(\alpha)}{N}\right\rfloor
\qquad (\alpha\in \oc_{\leq N-1}).
\]
Therefore $2S_{N-1}(\alpha)-N\ell(\alpha)={\rm rem}(\alpha)$ where ${\rm rem}(\alpha)$ is the remainder when $2S_{N-1}(\alpha)$ is divided by $N$, as in \eqref{eq:fnweighted}. Consequently $\displaystyle f_{N-1}(a^2x;1,(ca^{-1})^N)=f_{N-1}(x;a,c^N).$ Since every expression above depends on $c$ only through $c^N$, we may rename the indeterminate $c^N$ as $b$. It follows that the denominator of $J(x)$ is $f_{N-1}(x;a,b)$.
  \end{proof}

  \begin{theorem}\label{thm:finalgf}
For $N\geq 2$, let $p_N(x;a,b)=1-ax-bx^N$, and define
\[
H(x):=\frac{1}{p_N(x;a,b)}*\frac{1}{p_N(x;a,b)}.
\]
Then
\[
H(x)=\frac{f_{N-1}(x;a,b)}{p_N(\sqrt{x};a,b)p_N(-\sqrt{x};a,b)f_{N-1}(x;a,-b)}.
\]
  \end{theorem}
  \begin{proof}
Let $J(x)$ be the unique formal power series satisfying $H(x)=(1-J(x))^{-1}.$ Write $J(x)=R(x)/S(x)$ with $R(x),S(x)\in \QQ(a,b)[x]$ coprime. By Lemma \ref{lem:finalAden}, one has $S(x)=f_{N-1}(x;a,b).$ Since
\[
H(x)=\frac{1}{1-J(x)}=\frac{S(x)}{S(x)-R(x)},
\]
and $\gcd(S(x),S(x)-R(x))=\gcd(S(x),R(x))=1,$ this expression is already in lowest terms. Hence the numerator of $H(x)$ is $f_{N-1}(x;a,b)$.

On the other hand, Theorem \ref{prop:denominator}, applied with $k=N$, shows that the denominator of $H(x)$ in lowest terms is $p_N(\sqrt{x};a,b)p_N(-\sqrt{x};a,b)f_{N-1}(x;a,-b).$ Combining the numerator and denominator of $H(x)$ yields
\begin{align*}
H(x)&=\frac{f_{N-1}(x;a,b)}{p_N(\sqrt{x};a,b)p_N(-\sqrt{x};a,b)f_{N-1}(x;a,-b)}.\qedhere  
\end{align*}
\end{proof}
  \subsection{Tilings without a central horizontal fault}
  We now determine the contribution from vertically fault-free tilings that do not have a central horizontal fault. Let
  \[
  U_k(x;a,b):=\sum_{n\geq 0}u_{k,n}(a,b)x^n,
  \]
  where $u_{k,n}(a,b)$ is the total weight of the vertically fault-free tilings of a $2k\times n$ rectangle with $k\times 1$ tiles and with no central horizontal fault. In view of the definition of $V_k(x;a,b)$ in Equation \eqref{eq:hkvk}, the generating function for all vertically fault-free tilings of a $2k\times n$ rectangle is given by
  \[
  W_k(x;a,b):=V_k(x;a,b)+U_k(x;a,b).
  \]

  \begin{lemma}\label{lem:decompose}
With $W_k(x;a,b)$ as above, we have
\[
F_k(x;a,b)=\frac{1}{1-W_k(x;a,b)}.
\]
  \end{lemma}
  \begin{proof}
Since every tiling of a $2k\times n$ rectangle corresponds to a unique finite sequence of vertically fault-free tilings whose widths add up to $n$, the lemma follows.
  \end{proof}

  We also need Graham's characterization of fault-free tilings. A tiling is \emph{fault-free} if it has neither a vertical fault nor a horizontal fault.

  \begin{theorem}[Graham {\cite[p.~125]{Graham1981}}]\label{thm:graham}
Let $k_1,k_2,m,n$ be positive integers with $\gcd(k_1,k_2)=1$ and $mn>k_1k_2$. A fault-free tiling of a $m\times n$ rectangle with $k_1\times k_2$ tiles exists if and only if the following conditions hold:
\begin{enumerate}
\item Each of $k_1$ and $k_2$ divides at least one of $m$ and $n$.
\item Each of $m$ and $n$ can be expressed as $uk_1+vk_2$ for positive integers $u$ and $v$ in at least two ways.
\item If $\{k_1,k_2\}=\{1,2\}$, then $(m,n)\neq (6,6)$.
\end{enumerate}
  \end{theorem}

  We also require the following structural result of Aggarwal and Ram \cite[Lemma~4]{MR4537778}.

  \begin{proposition}\label{prop:ARstructure}
Suppose $k<m<2k$ and $n>k$. In any vertically fault-free tiling of an $m\times n$ rectangle by $k\times 1$ tiles, there exist $k$ contiguous rows such that all tiles in the remaining $m-k$ rows are horizontal.
  \end{proposition}

  We also record here the additional parts of Aggarwal and Ram's argument that we use below. Let $k>1$, let $k<m<2k$, and let $\ell\geq 1$. Then \cite[Proposition~5]{MR4537778} gives
  \[
  \#\{\text{vertically fault-free tilings of an $m\times k\ell$ rectangle}\}=
  \begin{cases}
  m-k+2, & \ell=1,\\[4pt]
  (m-k+1)\binom{k+\ell-3}{\ell-1}, & \ell\geq 2.
  \end{cases}
  \]
When $\ell=1$, there is exactly one tiling in which all tiles are horizontal and the other $m-k+1$ tilings contain precisely $k$ vertical tiles. For $\ell>1$, every vertically fault-free tiling contains precisely $m-k$ horizontal rows; after deleting those rows, the remaining $k\times k\ell$ rectangle corresponds bijectively with compositions $(\alpha_1,\ldots,\alpha_r)$ of $k\ell$ whose parts lie in $\{1,k\}$ and for which no proper partial sum is divisible by $k$. The bijection sends vertical tiles to 1, while a block of $k$ horizontal tiles maps to $k$. An example is shown in Figure \ref{fig:k3m5ell4}. Equivalently, we have $\alpha_1=\alpha_r=1,$ exactly $k$ parts of the composition $(\alpha_1,\ldots,\alpha_r)$ are equal to $1$, and every remaining part is equal to $k$. In particular, every vertically fault-free tiling of an $m\times k\ell$ rectangle with $\ell>1$ contains precisely $k$ vertical tiles \cite[Corollary~6]{MR4537778}.

  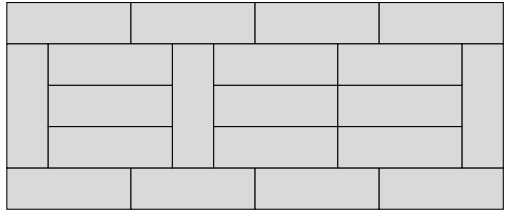
\begin{figure}[ht]
    \centering
    \begin{tikzpicture}[scale=0.55]
      \tikzset{tile/.style={fill=gray!30,draw=black}}
      \newcommand{\hTileThree}[2]{\fill[tile] (#1,#2) rectangle ++(3,1);}
      \newcommand{\vTileThree}[2]{\fill[tile] (#1,#2) rectangle ++(1,3);}

      \foreach \x in {0,3,6,9}{
        \hTileThree{\x}{4}
        \hTileThree{\x}{0}
      }

      \vTileThree{0}{1}
      \vTileThree{4}{1}
      \vTileThree{11}{1}
      \foreach \y in {1,2,3}{
        \hTileThree{1}{\y}
        \hTileThree{5}{\y}
        \hTileThree{8}{\y}
      }

    \end{tikzpicture}
    \caption{A vertically fault-free tiling for $k=3$, $m=5$, and $n=12$, with composition $(1,3,1,3,3,1)$.}
    \label{fig:k3m5ell4}
  \end{figure}

  \begin{proposition}[Klarner's divisibility criterion {\cite[Thm.~5]{MR248643}}]\label{prop:klarner-divisibility}
Let $m,n,k$ be positive integers. The $m\times n$ rectangle can be tiled by $k\times 1$ tiles if and only if at least one of $m$ and $n$ is divisible by $k$.
  \end{proposition}

\begin{proposition}\label{prop:noncentral}
The generating function for vertically fault-free tilings of a $2k\times n$ rectangle that do not have a central horizontal fault is
\[
U_k(x;a,b)=\frac{(k-1)a^kb^kx^k}{(1-b^{2k}x^k)^{k-1}}.
\]
  \end{proposition}
  We divide the proof into two lemmas.

  \begin{lemma}\label{lem:noncentralshape}
Let $T$ be a vertically fault-free tiling of a $2k\times n$ rectangle by $k\times 1$ tiles with no central horizontal fault. Then $n=k\ell$ for some $\ell\geq 1$, and there exists an integer $s$ with $1\leq s\leq k-1$ such that rows $s+1,\ldots,s+k$ contain every non-horizontal tile of $T$, while the top $s$ rows and the bottom $k-s$ rows are tiled entirely by horizontal tiles.
  \end{lemma}
  \begin{proof}
Apply Theorem \ref{thm:graham} with $(m,n,k_1,k_2)=(2k,n,k,1)$. The second condition fails because $2k=1\cdot k+k\cdot 1$ is the only expression of $2k$ as $uk+v$ with $u,v>0$. Hence no tiling of a $2k\times n$ rectangle by $k\times 1$ tiles is fault-free. Since $T$ is vertically fault-free, it must therefore have a horizontal fault. Because this fault is not central, reflecting $T$ in a horizontal axis if necessary allows us to assume that $T$ has a horizontal fault at a distance $r$ from the top of the $2k\times n$ rectangle for some $1\leq r\leq k-1$.

The top $r$ rows of $T$ are tiled entirely by horizontal tiles since $r<k$. The remaining rows therefore form a vertically fault-free tiling of a $(2k-r)\times n$ rectangle, and
\[
k<2k-r<2k.
\]
Since $k\nmid (2k-r)$, Klarner's criterion implies that $n=k\ell$ for some $\ell\geq 1$.

If $\ell=1$, then the remaining $(2k-r)\times k$ rectangle is vertically fault-free and is not tiled entirely by horizontal bars, since otherwise $T$ itself would be completely horizontal and would have a central horizontal fault. By the aforementioned $\ell=1$ case of Aggarwal--Ram's Proposition~5, this remaining rectangle therefore contains exactly $(2k-r)-k=k-r$ horizontal rows and $k$ vertical tiles. Because a horizontal $k\times 1$ tile in a width-$k$ rectangle occupies an entire row, the rows not occupied by those $k-r$ horizontal rows form a unique block of $k$ consecutive rows filled by the vertical tiles. Together with the top $r$ horizontal rows already identified, this shows that all non-horizontal tiles of $T$ lie in a block of the form rows $s+1,\ldots,s+k$. Since $T$ has no central horizontal fault, this block is neither the top $k$ rows nor the bottom $k$ rows of the full rectangle. Thus it is of the form rows $s+1,\ldots,s+k$ with $1\leq s\leq k-1$ and the lemma follows.

If $\ell>1$, then Proposition \ref{prop:ARstructure} applied to the $(2k-r)\times k\ell$ rectangle shows that there exists a strip of $k$ contiguous rows such that all rows outside the strip are horizontal. Since $T$ has no central horizontal fault, this strip cannot be the top $k$ rows or the bottom $k$ rows of the full $2k\times k\ell$ rectangle. Therefore it is again of the form rows $s+1,\ldots,s+k$ with $1\leq s\leq k-1$ and the top $s$ rows and bottom $k-s$ rows are tiled entirely by horizontal tiles. This proves the lemma.
  \end{proof}

  \begin{lemma}\label{lem:stripcompositions}
Fix $\ell\geq 1$. Consider tilings of a $k\times k\ell$ rectangle by $k\times 1$ tiles that contain at least one vertical tile and have no vertical fault at any line $x=jk$ for $1\leq j\leq \ell-1$. The number of such tilings is
\[
\binom{k+\ell-3}{\ell-1}.
\]
  \end{lemma}
  \begin{proof}
By the quoted description from the proof of Aggarwal--Ram's Proposition~5, the tilings in the statement are exactly the $k\times k\ell$ cores corresponding under the standard bijection to compositions $(\alpha_1,\ldots,\alpha_t)$ of $k\ell$ with parts in $\{1,k\}$ such that $\alpha_1=\alpha_t=1$ with exactly $k$ parts equal to $1$ and every remaining part equal to $k$.

Such a composition has $k+\ell-1$ parts in total. After fixing the first and last parts equal to $1$, the remaining $k-2$ parts equal to $1$ may be placed arbitrarily among the $k+\ell-3$ internal positions. Therefore the number of admissible compositions is
\[
\binom{k+\ell-3}{k-2}=\binom{k+\ell-3}{\ell-1}.
\]
This proves the lemma.
  \end{proof}

  \begin{proof}[Proof of Proposition \ref{prop:noncentral}]
Fix a vertically fault-free tiling $T$ of a $2k\times n$ rectangle with no central horizontal fault. By Lemma \ref{lem:noncentralshape}, we may write $n=k\ell$ with $\ell\geq 1$, and there exists $s\in\{1,\ldots,k-1\}$ such that the top $s$ rows and bottom $k-s$ rows are horizontal while all non-horizontal tiles lie in the $k$ middle rows.

Removing those $k$ horizontal rows leaves a tiling of a $k\times k\ell$ rectangle. Because $T$ is vertically fault-free, this middle strip has no vertical fault at any line $x=jk$ for $1\leq j\leq \ell-1$. It also contains at least one vertical tile, for otherwise $T$ would be tiled entirely by horizontal bars and would have a central horizontal fault. Hence Lemma \ref{lem:stripcompositions} applies and gives
\[
\binom{k+\ell-3}{\ell-1}
\]
possible middle strips for each fixed value of $s$.

Moreover, every admissible composition in Lemma \ref{lem:stripcompositions} has exactly $k$ parts equal to $1$, so the corresponding middle strip contains $k$ vertical tiles, each spanning all $k$ rows. Thus every row of the strip contains a vertical tile, and the strip position $s$ is uniquely determined by $T$. Consequently the tilings counted by $u_{k,k\ell}(a,b)$ are in bijection with pairs $(s,\alpha)$ where $1\leq s\leq k-1$ and $\alpha$ is an admissible composition counted by Lemma \ref{lem:stripcompositions}.

For such a tiling, the middle strip contributes $k$ vertical tiles and $(\ell-1)k$ horizontal tiles, while the $k$ horizontal rows outside the strip contribute $k\ell$ horizontal tiles. Hence every such tiling has weight $a^kb^{(\ell-1)k+k\ell}=a^kb^{(2\ell-1)k}.$ It follows that
\[
u_{k,k\ell}(a,b)=(k-1)\binom{k+\ell-3}{\ell-1}a^kb^{(2\ell-1)k}
\qquad (\ell\geq 1),
\]
and $u_{k,n}(a,b)=0$ when $k\nmid n$. Therefore
\[
U_k(x;a,b)=\sum_{\ell\geq 1}(k-1)\binom{k+\ell-3}{\ell-1}a^kb^{(2\ell-1)k}x^{k\ell}.
\]
Setting $j=\ell-1$, we obtain
\[
U_k(x;a,b)=(k-1)a^kb^kx^k\sum_{j\geq 0}\binom{k+j-2}{j}(b^{2k}x^k)^j.
\]
The stated formula now follows from the binomial-series identity
\begin{align*}
\sum_{j\geq 0}\binom{r+j-1}{j}z^j&=(1-z)^{-r}. \qedhere
\end{align*}
  \end{proof}
We are now ready to prove Theorem \ref{thm:main} of the introduction.
  \begin{theorem}\label{thm:mainproof}
For each integer $k>1$, the generating function $F_k(x;a,b)$ is given by
\[
\frac{(1 - b^{2k}x^k)^{k-1}f_{k-1}(x;a,b^k)}
     {p_k(\sqrt{x};a,b^k)p_k(-\sqrt{x};a,b^k)(1 - b^{2k}x^k)^{k-1}f_{k-1}(x;a,-b^k)-(k-1)a^kb^kx^kf_{k-1}(x;a,b^k)},
\]
where $p_k(x;a,b)=1-ax-bx^k$.
  \end{theorem}
  \begin{proof}
Let $W_k(x;a,b)=V_k(x;a,b)+U_k(x;a,b)$ be the generating function for all vertically fault-free tilings of a $2k\times n$ rectangle. By Lemma \ref{lem:decompose},
\[
F_k(x;a,b)=\frac{1}{1-W_k(x;a,b)}.
\]
By Proposition \ref{prop:Hkvkandhadamard},
\[
\frac{1}{1-V_k(x;a,b)}=\frac{1}{p_k(x;a,b^k)}*\frac{1}{p_k(x;a,b^k)}.
\]
Applying Theorem \ref{thm:finalgf} with $N=k$ and with $b$ replaced by $b^k$, we obtain
\[
\frac{1}{1-V_k(x;a,b)}
=\frac{f_{k-1}(x;a,b^k)}{p_k(\sqrt{x};a,b^k)p_k(-\sqrt{x};a,b^k)f_{k-1}(x;a,-b^k)}.
\]

Combining this identity with Proposition \ref{prop:noncentral} yields
\[
1-W_k(x;a,b)=\frac{p_k(\sqrt{x};a,b^k)p_k(-\sqrt{x};a,b^k)f_{k-1}(x;a,-b^k)}{f_{k-1}(x;a,b^k)}
-\frac{(k-1)a^kb^kx^k}{(1-b^{2k}x^k)^{k-1}}.
\]
Taking reciprocals and clearing denominators gives the desired formula. 
  \end{proof}
   
\section{Acknowledgements}
S. R. acknowledges with appreciation support from an Indo-Russian project, grant number DST/INT/RUS/RSF/P41/2021.
\printbibliography  
\end{document}